\documentclass[12pt]{amsart}
\usepackage{amsmath,eucal,verbatim,amsopn,amsthm,amsfonts,amssymb,mathdots,mathrsfs,esint,mathtools,enumitem,mnsymbol}
\usepackage[margin=1in]{geometry}
\let\mathcal=\CMcal
\allowdisplaybreaks[4]

\newcommand{\setdiff}{\ensuremath{-}}
\newcommand{\setof}[2]{\ensuremath{\{ #1 : #2 \}}}
\newcommand{\isomorphic}{\ensuremath{\cong}}

\newcommand{\mathe}{\ensuremath{\mathrm{e}}}
\newcommand{\transpose}[1]{\ensuremath{\rconj{t}#1}}

\def\Hom{\mathrm{Hom}}
\def\Tri{\mathrm{Tri}}
\def\Dim{\mathrm{dim}}

\newcommand{\nonarchimedean}{non-Archimedean} 
\newcommand{\GL}{\ensuremath{\mathrm{GL}}}
\newcommand{\Sp}{\ensuremath{\mathrm{Sp}}}
\newcommand{\SL}{\ensuremath{\mathrm{SL}}}
\newcommand{\GSpin}{\ensuremath{\mathrm{GSpin}}}
\newcommand{\SO}{\ensuremath{\mathrm{SO}}}
\newcommand{\Spin}{\ensuremath{\mathrm{Spin}}}
\newcommand{\rmodulo}[2]{#1 / #2} 
\newcommand{\glnidx}{\ensuremath{n}}

\newcommand{\half}{\ensuremath{\frac12}}

\newcommand{\lmodulo}[2]{\ensuremath{#1 \backslash #2}}
\newcommand{\rconj}[1]{\ensuremath{{}^{#1}}}

\def\Ind{\operatorname{Ind}}
\def\ind{\operatorname{ind}}
\newcommand{\C}{\ensuremath{\mathbb{C}}}

\newcommand{\Adele}{\ensuremath{\mathbb{A}}}
\newcommand{\rhs}{right-hand side} 
\newcommand{\lhs}{left-hand side} 

\newcommand{\cover}[1]{\ensuremath{\widetilde{#1}}}
\newcommand{\setdifference}{\ensuremath{\mathrm{-}}}

\newtheorem{theorem}{Theorem}

\newtheorem{lemma}{Lemma}[section]
\newtheorem{proposition}[lemma]{Proposition}
\newtheorem{corollary}[lemma]{Corollary}
\newtheorem{claim}[lemma]{Claim}
\newtheorem{remark}{Remark}[section]

\numberwithin{equation}{section}

\begin{document}
\title[]{Theta distinguished representations, inflation and the symmetric square $L$-function}
\author{Eyal Kaplan}
\email{kaplaney@gmail.com}

\maketitle
\begin{abstract}
Let $\Pi_0$ be a representation of a group $H$. We say that a representation $\tau$ is $(H,\Pi_0)$-distinguished, if it is
a quotient of $\Pi_0$. It is natural to ask whether this notion ``inflates" to larger groups, in the sense that a representation $\mathrm{I}(\tau)$
induced from $\tau$ and $H$ to a group $G$, is $(G,\Pi)$-distinguished. We study representations distinguished by theta representations:  
$H=\GL_n$, $\Pi_0$ is a pair of the exceptional representations of Kazhdan and Patterson, $G=\GSpin_{2n+1}$ and $\Pi$ is a pair of 
the small representations of Bump, Friedberg and Ginzburg. We prove a Rodier-type hereditary property: a tempered representation $\tau$ is distinguished
if and only if $\mathrm{I}(\tau)$ is distinguished, and the multiplicity in each model is the same.
If $\tau$ is supercuspidal and distinguished, we prove that the Langlands quotient of $\mathrm{I}(\tau)$ is distinguished. As a corollary, we
characterize supercuspidal distinguished representations, in terms of the pole of the local symmetric square $L$-function at $s=0$.
\end{abstract}

\section{Introduction}\label{section:introduction}
Let $\tau$ be an admissible representation of $\GL_n(F)$, where $F$ is a local non-Archimedean field.
Let $\theta_0$ and $\theta_0'$ be a pair of exceptional representations of a double cover $\widetilde{\GL}_n(F)$ of $GL_n(F)$, in the sense of Kazhdan and Patterson \cite{KP}.
We say that $\tau$ is distinguished if
\begin{align}\label{eq:homspace dist}
\Hom_{\GL_n(F)}(\theta_0\otimes\theta_0',\tau^{\vee})\ne0.
\end{align}
Here $\tau^{\vee}$ is the representation contragradient to $\tau$. Equivalently, the space 
of $\GL_n(F)$-invariant trilinear forms on $\tau\times\theta_0\times\theta_0'$ is nonzero.

This space first appeared in a global context.
Let $\pi$ be a unitary cuspidal automorphic representation of $\GL_n(\Adele)$, where $\Adele$ is the
ad\`{e}les ring of a global field. Assume that $\pi$ has a trivial central character. Bump and Ginzburg \cite{BG} proved that if the partial symmetric square $L$-function
$L^S(s,\pi,\mathrm{Sym}^2)$ has a pole at $s=1$, the following period integral is nonvanishing
\begin{align}\label{int:BG period}
\int_{Z\GL_n(F)\backslash\GL_n(\Adele)}\varphi_{\pi}(g)\Theta(g)\Theta'(g)dg.
\end{align}
Here $Z$ is a subgroup of finite index in the center $C_{\GL_n}(\Adele)$ of $\GL_n(\Adele)$ ($Z=C_{\GL_n}(\Adele)$ when $n$ is odd);
$\varphi_{\pi}$ is a cusp form in the space of $\pi$; $\Theta$ and $\Theta'$ are automorphic forms in the space of a global exceptional representation of $\widetilde{\GL}_n(\Adele)$. For $n=2$, an earlier work
by Patterson and Piatetski-Shapiro \cite{PPS} showed that a similar integral characterizes the pole at $s=1$, for a global field of odd characteristic.

Periods of automorphic forms are often related to poles of $L$-functions and to questions of functoriality. Ginzburg, Jiang and Soudry \cite{GJS} described such relations in a general setup
and also considered several examples. Let $G_n=\GSpin_{2n+1}$ be the odd general spin group of rank $n+1$. Let $E(g;\rho,s)$ be the
Eisenstein series corresponding to an element $\rho$ in the space of the representation of
$G_n(\Adele)$ induced from $\tau|\det{}|^s\otimes1$ ($s\in\C$, $g\in G_n(\Adele)$). The residual representation
$E_{\pi}$ is the space spanned by the residues $E_{1/2}(\cdot;\rho)$ of $E(g;\rho,s)$ at $s=1/2$. The following result formulated in
\cite{GJS} (Theorem~3.3) was proved in a series of works (\cite{BG,BFG,GJS,me7}):
the following conditions are equivalent.
\begin{enumerate}
\item $L^S(s,\pi,\mathrm{Sym}^2)$ has a pole at $s=1$.
\item\label{item:period} The period integral \eqref{int:BG period} is nonzero.
\item\label{item:residual} The residual representation $E_{\pi}$ is nonzero.
\item $\pi$ is the Langlands functorial transfer of an irreducible generic cuspidal automorphic representation of the split $\mathrm{SO}_{n}(\Adele)$ (if $n$ is even) or $\Sp_{(n-1)/2}(\Adele)$ ($n$ is odd).
\end{enumerate}
Note that this was stated in \cite{GJS} with $G_n$ replaced by $\SO_{2n+1}$ (see below).

To prove that the nonvanishing of \eqref{int:BG period} implies the nontriviality of $E_{\pi}$, one can relate the period to the following co-period
integral
\begin{align*}
\mathcal{CP}(E_{1/2}(\cdot;\rho),\Theta,\Theta')=\int_{C_{G_{n}(\Adele)}G_{n}(F)\backslash G_{n}(\Adele)}E_{1/2}(g;\rho)\Theta(g)\Theta'(g)dg.
\end{align*}
Here $\Theta$ and $\Theta'$ belong to the exceptional, or small, representation of Bump, Friedberg and Ginzburg \cite{BFG},
whose analog for $G_n$ was described in \cite{me8}. This integral was studied in \cite{me7} for $\SO_{2n+1}$ (extended in \cite{me8} to $G_n$) and we proved
the implication $\eqref{item:period}\Rightarrow\eqref{item:residual}$. In the setting of $G_n$, one can also study the twisted symmetric square $L$-function. The Rankin-Selberg integral representation
for this function has recently been developed by Takeda \cite{Tk}.

The global unfolding of $\mathcal{CP}(E_{1/2}(\cdot;\rho),\Theta,\Theta')$ (in \cite{me7}) has a local counterpart. Assume that $\tau$ is a distinguished
representation. Let $\mathrm{I}(\tau)=\Ind_{Q_n}^{G_n}(\delta_{Q_n}^{1/2}\tau|\det{}|^{1/2}\otimes1)$, where $Q_n$ is the Siegel parabolic subgroup. Using an integral over $Q_n(F)\backslash G_n(F)$, we show (Proposition~\ref{proposition:upper heredity of dist 1})
that for a certain pair $\theta$ and $\theta'$ of exceptional representations of $\widetilde{G}_n(F)$ (\cite{BFG,me8}, see below),
\begin{align}\label{eq:homspace dist2}
\Hom_{G_n(F)}(\theta\otimes{\theta'},\mathrm{I}(\tau)^{\vee})\ne0.
\end{align}
In other words, $\mathrm{I}(\tau)$ is a distinguished representation of $G_n(F)$. In fact, depending on the central character of $\tau$, we may
need to replace $\tau|\det{}|^{1/2}\otimes1$ with $\tau|\det{}|^{1/2}\otimes\eta$ where $\eta$ is a character of $F^*$. (See the proposition for details.)
The local problem is to prove that the Langlands quotient
$\mathrm{LQ}(\mathrm{I}(\tau))$ of $\mathrm{I}(\tau)$ is also distinguished.

In the present study we consider an irreducible unitary supercuspidal $\tau$.
In this case $\mathrm{I}(\tau)$ is either irreducible and generic, or is of length two, has a unique irreducible
generic subrepresentation and $\mathrm{LQ}(\mathrm{I}(\tau))$ is non-generic.
Here is our main result, implying that if $\tau$ is distinguished, so is $\mathrm{LQ}(\mathrm{I}(\tau))$. For more details see Corollary~\ref{corollary:upper hered dist LQ}.
\begin{theorem}\label{theorem:tensor of small is non generic}
The space of $\theta\otimes\theta'$ as a representation of $G_n(F)$ does not afford a Whittaker functional.
\end{theorem}

A similar ``inflation" phenomena has already been observed by Ginzburg, Rallis and Soudry \cite{GRS5} (Theorem~2). Assume that $\tau$ is a supercuspidal and self-dual representation, such that the exterior square $L$-function $L(s,\tau,\wedge^2)$ has
a pole at $s=0$. This implies that $\tau$ has a Shalika model and then according to Jacquet and Rallis \cite{JR2}, $\tau$ admits a (nontrivial) $\GL_n\times\GL_n$ invariant functional.
In turn, the representation 
parabolically induced from $\tau|\det|^{1/2}$ to $\Sp_{2n}$ has an $\Sp_n\times\Sp_n$ invariant functional.
Ginzburg, Rallis and Soudry \cite{GRS7} (Theorems~16-17) showed that an irreducible generic representation of $\Sp_{2n}$ does not admit such
a functional. It follows that the $\Sp_n\times\Sp_n$ functional factors through the Langlands quotient.

This inflation was one of the ingredients used
by Lapid and Mao for the proof of their conjecture on
Whittaker-Fourier coefficients, in the case of the metaplectic group (\cite{LM3,LM6,LM4,LM5}). Note that their conjecture actually applies
to any quasi-split group as well as the metaplectic group.
Our results here and their extension, in a forthcoming work, to an arbitrary irreducible generic distinguished representation,
are expected to be used in a proof of this conjecture for even orthogonal groups.

Here we have a similar relation between distinguished representations and the symmetric square $L$-function.
\begin{theorem}\label{theorem:supercuspidal dist GLn and pole}
Let $\tau$ be an irreducible unitary supercuspidal representation of $\GL_n$. Then $\tau$ is
distinguished if and only if $L(s,\tau,\mathrm{Sym}^2)$ has a pole at $s=0$.
\end{theorem}
Refer to Shahidi \cite{Sh5} (Theorem~6.2) for a description of the poles of $L(s,\tau,\mathrm{Sym}^2)$ in this setting.

The proof of Theorem~\ref{theorem:tensor of small is non generic} is essentially the local analog of the global computation of the co-period
in \cite{me7,me8}. Write $Q_n=M_n\ltimes U_n$ and let $C=C_{U_n}$ be the center of the unipotent radical $U_n$. The global unfolding argument involves a Fourier expansion
of $\Theta$ along $C$. Consider a non-generic character of $C$, this means that its stabilizer in $M_n$ contains a unipotent radical $V$ of a parabolic subgroup of $\GL_n$. The corresponding Fourier coefficient
is constant on $V$, then the cuspidality of $\pi$ is used to prove that the integral vanishes.

When $n$ is odd (and $n>1$), all characters of $C$ are non-generic. In the even case there is one generic orbit of characters. Its stabilizer
is ``almost" a Jacobi group, its reductive part is $Sp_{n/2}\times G_0$, where $Sp_{n/2}$ is a symplectic group in $n$ variables. One might attempt
to prove invariance of the product of Fourier coefficients under $Sp_{n/2}(\Adele)$, then use the fact that $\pi$ does not admit nontrivial
symplectic periods (\cite{JR}). Albeit the Fourier coefficients do not enjoy this invariance, a certain convolution against Weil theta
functions, introduced by Ideka \cite{Ik3}, can be used instead. 


The local argument involves the computation of the twisted Jacquet modules $\theta_{C,\psi_k}$ of $\theta$ with respect to $C$ and
a representative $\psi_k$ of some orbit of characters. In contrast with the global setting, the generic character is the
crux of the proof. Roughly, this is because
$\theta_{C,\psi_k}$ is a tensor of an exceptional representation of $\widetilde{\GL}_k(F)$ and the Jacquet module of an exceptional representation of $\widetilde{G}_{2k}(F)$ along $C_{U_{2k}}$ and a generic character.

When the character is generic, the twisted Jacquet module is (by restriction) a representation of a Jacobi group.
The local theory of smooth representations of Jacobi groups \cite{Dijk,MVW,BSch} describes such a representation as a tensor $\kappa\otimes\omega_{\psi}$,
where $\omega_{\psi}$ is the Weil representation. The Heisenberg group acts trivially on the space of $\kappa$, while the action of the reductive part separates into
an action on the space of $\kappa$, and one on the space of $\omega_{\psi}$. We prove that $\kappa$ is a trivial representation of $Sp_{n/2}(F)$.
\begin{theorem}\label{theorem:small rep is weakly minimal}
Assume $n$ is even and let $\psi_{n/2}$ be a generic character of $C_{U_n}$. As a Jacobi representation $\theta_{C_{U_n},\psi_{n/2}}$ is the direct sum of (possibly infinitely many) copies of $\omega_{\psi}$.
\end{theorem}
Note that the action of the $G_0$ part of the stabilizer on $\theta_{C_{U_n},\psi_{n/2}}$ is given simply by the central character of $\theta$.

This result underlies Theorem~\ref{theorem:tensor of small is non generic}. Furthermore, it implies the following multiplicity property.
Assume that $\tau$ is irreducible and tempered. We prove that $\tau$ is distinguished if and only if $\mathrm{I}(\tau)$ is.
Moreover, the dimensions of \eqref{eq:homspace dist2} and \eqref{eq:homspace dist} are equal. See Proposition~\ref{proposition:one-dim}
for the precise statement. Note that Kable \cite{Kable} conjectured, and under a certain homogeneity assumption
proved (\cite{Kable} Corollary~6.1), that \eqref{eq:homspace dist} enjoys multiplicity one. These results motivate the introduction of ``exceptional models" over $\GL_n$ and $G_n$.

Theorem~\ref{theorem:small rep is weakly minimal} may have additional applications. Explicit descriptions of Jacquet modules of exceptional representations have had numerous applications (see below).

Note that when $n=2$, $Q_n$ is the Heisenberg parabolic subgroup in the notation of Gan and Savin \cite{GanSavin}. In their terminology, Theorem~\ref{theorem:small rep is weakly minimal}
shows that $\theta$ is ``weakly minimal". In this case it is the minimal representation and $\theta_{C_{U_n},\psi_{n/2}}\isomorphic\omega_{\psi}$ (\cite{GanSavin} Section~3).

Bump, Friedberg and Ginzburg \cite{BFG} constructed the small representation $\theta^{\SO_{2n+1}}$ for the special odd orthogonal group. This is a representation of
a ``double cover" $\widetilde{SO}_{2n+1}(F)$, obtained by restricting the $4$-fold
cover of $\SL_{2\glnidx+1}(F)$ of Matsumoto \cite{Mats}. For the low rank cases $n=2,3$, it is the minimal representation. In fact for $n=3$,
this representation was already developed by Roskies \cite{Roskies}, Sabourin \cite{Sabourin} and Torasso \cite{Torasso}.
For $n>3$, there is no minimal representation for a group of type $B_n$ (\cite{V}).

Bump, Friedberg and Ginzburg \cite{BFG2} showed that when $n>3$, $\theta^{\SO_{2n+1}}$ is attached to one of the
possible coadjoint orbits, which is smallest next to the minimal one. This translates into the vanishing of a large class
of Fourier coefficients, called generic in \cite{BFG} (see also \cite{CM,Cr,G2}). Locally, this means that a large class
of twisted Jacquet modules vanish. The representation $\theta^{\SO_{2n+1}}$ was used by these authors to construct a lift with certain
functorial properties, between covers of orthogonal groups \cite{BFG2}. The representation $\theta^{SO_{7}}$ was used to construct an integral representation
(\cite{BFG3}).

There is a technical issue when working with $\widetilde{SO}_{2n+1}(F)$: the underlying field must contain all $4$ $4$-th roots of unity.
This can be remedied using $G_n$. Indeed, one obtains a nontrivial double cover of $G_n(F)$ by restricting the $2$-fold cover of $Spin_{2n+3}(F)$
of Matsumoto \cite{Mats}. The theory of Bump, Friedberg and Ginzburg \cite{BFG} can be extended to $\widetilde{G}_n(F)$, mainly because
both groups are of type $B_n$ and in particular, the unipotent subgroups are isomorphic. The details were carried out in \cite{me8}.
Our results here are stated for $G_n$, but apply similarly to $\SO_{2n+1}$ and $\theta^{\SO_{2n+1}}$ (see Section~\ref{section:relevance to SO(2n+1)}).

Minimal representations have been studied and used extensively, by numerous authors.
Knowledge of Fourier coefficients, or Jacquet modules, has proved very useful for applications
\cite{GRS6,G2,GJS2}. They have played a fundamental role in the theta correspondence, the descent method and Rankin-Selberg integrals.
See, for example \cite{V,KZ,KS,Pd,Savin2,BK,Savin,GRS,BFG3,GRS3,KP3,BFG,JSd1,GanSavin,Soudry4,LokeSavin,LokeSavin2,RGS}.


We mention that Bump and Ginzburg \cite{BG} also developed a local theory, where they considered a similar space of equivariant trilinear forms, except that
$\theta_0'$ was replaced by a certain induced representation. 
For $n=3$, Savin \cite{Savin3} determined the dimension
of \eqref{eq:homspace dist} for an arbitrary irreducible quotient of a principal series representation. He also conjectured (and proved for $n=3$),
that the class of distinguished spherical representations of $\GL_n(F)$ is precisely those representations, which are lifts from
a certain prescribed classical group. Kable \cite{Kable2} proved that these lifts are distinguished, the other
direction was proved in \cite{me9}.

The group $G_n$ has been the focus of study of a few recent works, among which are the works of
Asgari \cite{Asg2,Asg} on local $L$-functions, Asgari and Shahidi \cite{AsgSha,AsgSha2} on functoriality
and Hundley and Sayag \cite{HS} on the descent construction.

The rest of this work is organized as follows. In Section~\ref{section:preliminaries} we provide notation and definitions. In particular we
describe the construction of exceptional representations. Section~\ref{section:Heisenberg jacquet modules} contains the proof of
Theorem~\ref{theorem:small rep is weakly minimal}. Theorems~\ref{theorem:tensor of small is non generic}
and \ref{theorem:supercuspidal dist GLn and pole} are proved in Section~\ref{section:Distinguished representations}.
Section~\ref{section:relevance to SO(2n+1)} provides the formulation of our results for $\SO_{2n+1}$.

\subsection*{Acknowledgments}
I wish to express my gratitude to Erez Lapid for suggesting this project to me. I would like to thank Eitan Sayag, for explaining
to me how to use his results with Offen (\cite{OS}, Proposition~1). 
Lastly, I thank Jim Cogdell for his kind encouragement and useful remarks.

\section{Preliminaries}\label{section:preliminaries}

\subsection{General notation}\label{subsection:general notation}
Let $F$ be a local \nonarchimedean\ field of characteristic different from $2$. 
For an integer $r\geq1$, let $\mu_r$ be the subgroup of $r$-th roots of unity in $F$. Set $F^{*r}=(F^*)^r$.
We usually fix a nontrivial additive character $\psi$ of $F$. Then the normalized Weil factor (\cite{We} Section~14)
is denoted by $\gamma_{\psi}$ ($\gamma_{\psi}(a)$ is
$\gamma_F(a,\psi)$ of \cite{Rao}, $\gamma_{\psi}(\cdot)^4=1$). The Hilbert symbol of order $r$ is $(,)_r$.
If $G$ is a group, $C_G$ denotes its center. For $x,y\in G$ and $Y<G$, $\rconj{x}y=xyx^{-1}$ and
$\rconj{x}Y=\setof{\rconj{x}y}{y\in Y}$. Hereby we omit references to the field, e.g., $\GL_n=\GL_n(F)$.

\subsection{The group $\GL_{n}$ and its cover}\label{subsection:GL_n and its cover}
Fix the Borel subgroup $B_{\GL_n}=T_{\GL_n}\ltimes N_{\GL_n}$ of upper triangular matrices, where $T_{\GL_n}$ is the diagonal torus. 
For any $k_1,k_2\geq0$ such that $k_1+k_2=n$, denote by $P_{k_1,k_2}$ 
the maximal parabolic subgroup whose Levi part is isomorphic to $\GL_{k_1}\times\GL_{k_2}$.  
Its unipotent radical is $Z_{k_1,k_2}=\{\left(\begin{smallmatrix}I_{k_1}&z\\&I_{k_2}\end{smallmatrix}\right)\}$. The ``mirabolic" subgroup 
$P_{n-1,1}^{\circ}$ is the subgroup of $\GL_n$ of matrices with the last row $(0,\ldots,0,1)$. 
Let $I_n$ be the identity matrix of $\GL_n$ and $J_{n}$ be the matrix with $1$ on the anti-diagonal and $0$ elsewhere. For $g\in\GL_{n}$, $\transpose{g}$ is the transpose of $g$.

We will use the metaplectic double cover $\widetilde{\GL}_n$ of $\GL_n$, as constructed by
Kazhdan and Patterson \cite{KP}.
Let $\widetilde{SL}_{n+1}$ be the double cover of $\mathrm{SL}_{n+1}$ of Matsumoto \cite{Mats} and let
$\sigma_{\mathrm{SL}_{n+1}}$ be the corresponding cocycle of Banks, Levi and Sepanski \cite{BLS} (Section~3).
We define a $2$-cocycle $\sigma_{\GL_n}$ of $\GL_{n}$ via
\begin{align*}
\sigma_{\GL_n}(a,a')=\sigma_{\mathrm{SL}_{n+1}}(\mathrm{diag}(\det{a}^{-1},a),\mathrm{diag}(\det{a'}^{-1},a')).
\end{align*}
This cocycle is related to the cocycle $\sigma_{n}$ of $\GL_n$ defined in \cite{BLS} by
\begin{align*}
\sigma_{\GL_n}(a,a')=c(\det{a},\det{a'})\sigma_{n}(a',a).
\end{align*}
In particular $\sigma_{\GL_1}(a,a')=(a,a')_2$.

\subsection{The group $\GSpin_{2n+1}$}\label{subsection:GSpin}
We start by defining the special odd orthogonal group
\begin{align*}
\SO_{2n+1}=\setof{g\in \SL_{2n+1}}{\transpose{g}J_{2n+1}g=J_{2n+1}}.
\end{align*}
Select its Borel subgroup $B_{\SO_{2n+1}}=B_{\GL_{2n+1}}\cap \SO_{2n+1}$.

Let $\Spin_{2n+1}$ be the simple split simply-connected algebraic group of type $B_{n}$. It is the algebraic double cover of $\SO_{2n+1}$. We will take the Borel subgroup, which is the preimage of $B_{\SO_{2n+1}}$. The set of simple roots of $\Spin_{2n+1}$ is
$\Delta_{\Spin_{2n+1}}=\setof{\alpha_i}{1\leq i\leq n}$, where $\alpha_i=\epsilon_i-\epsilon_{i+1}$ for $1\leq i\leq n-1$ and $\alpha_{n}=\epsilon_{n}$.

The group $G_n=\GSpin_{2n+1}$ is an $F$-split connected reductive algebraic group, which can be defined using a based root
datum as in \cite{Asg,AsgSha,HS}. It is also embedded in $G'_{n+1}=\Spin_{2n+3}$ as the Levi part of the parabolic subgroup
corresponding to $\Delta_{G'_{n+1}}\setdifference\{{\alpha_1}\}$ (see \cite{Mt}). We adapt this identification, which
is more natural for the purpose of cover groups, because we will obtain a cover of $G_n$ by restricting
the cover of $G'_{n+1}$.

Let $\delta_{Q}$ be the modulus character of a parabolic subgroup $Q<G_n$.

In the degenerate case $G_0=\GL_{1}$.

The Borel subgroup of $G_n$ is denoted $B_n=T_{n+1}\ltimes N_{n}$, where
$N_{n}$ is the unipotent radical (the rank of the torus is $n+1$). For $0\leq k\leq n$, denote by $Q_k=M_k\ltimes U_k$ the standard maximal parabolic subgroup of $G_{n}$ with $M_k\cong\GL_{k}\times G_{n-k}$. This isomorphism is not canonical. We describe the choice used in \cite{me8}, which is convenient for certain computations (see below).

The derived group $\SL_k$ of $\GL_{k}$ is generated by the root subgroups of $\setof{\alpha_i}{2\leq i\leq k}$. Let $\eta_1^{\vee},\ldots,\eta_{k}^{\vee}$ be the standard cocharacters of
$T_{\GL_k}$ and map $\eta_i^{\vee}\mapsto \epsilon_{i+1}^{\vee}-\epsilon_{1}^{\vee}$ for $1\leq i\leq k$.
Regarding $G_{n-k}$, the set $\setof{\alpha_i}{k+2\leq i\leq n+1}$ identifies $G'_{n-k}$ and if $\theta_1,\ldots,\theta_{n-k+1}$ are the characters of $T_{n-k+1}$, define $\theta_1^{\vee}\mapsto\epsilon_1^{\vee}$ and for $2\leq i\leq n-k+1$, $\theta_i^{\vee}\mapsto \epsilon_{k+i}^{\vee}$.

The projection $G'_{n}\rightarrow \SO_{2n+1}$ is an isomorphism between unipotent subgroups, hence
we can identify unipotent subgroups of $G_{n}$ with those of $\SO_{2n+1}$.

We use the character $\epsilon_1$ to define a ``canonical" character $\Upsilon$ of $G_{n}$. Namely $\Upsilon$ is the extension of $-\epsilon_1$ to $G_{n}$ (the only other choice would be to use $\epsilon_1$).

The aforementioned embedding of $\GL_k\times G_{n-k}$ in $M_k$ has a few properties, suitable for computations. To compute the conjugation of $b\in\GL_k$ on $U_k$, we can simply look at this action in $\SO_{2n+1}$, where $b$ takes the
form $\mathrm{diag}(b,I_{2(n-k)+1},J_k\transpose{b}^{-1}J_k)$. The image of $G_0$ is $C_{G_n}$.
The restriction of $\Upsilon$ to $\GL_k$ is $\det$.

\subsection{The double cover of $\GSpin_{2n+1}$}\label{subsection:metaplectic GSpin}
Let $\cover{G}'_{n+1}$ be the double cover of $G'_{n+1}$, constructed by Matsumoto \cite{Mats} using $(,)_{2}$ as the Steinberg symbol. Restricting the cover to $G_{n}$, we obtain the exact sequence
\begin{align*}
1\rightarrow{\mu_2}
\rightarrow\cover{G}_{n}\xrightarrow{p}G_{n}\rightarrow 1.
\end{align*}
Then $\cover{G}_{n}$ is a double cover of $G_{n}$. For a subset $X\subset G_n$, $\widetilde{X}=p^{-1}(X)$.
Let $\mathfrak{s}:G'_{n+1}\rightarrow\cover{G}'_{n+1}$ be the block-compatible section constructed by Banks, Levi and Sepanski \cite{BLS}
and $\sigma_{G'_{n+1}}$ be the corresponding cocycle. Denote the restriction of $\sigma_{G'_{n+1}}$ to $G_n\times G_n$ by $\sigma_{G_n}$. In \cite{me8} we proved that $\sigma_{G_n}$ satisfies the following block compatibility property: if $(a,g),(a',g')\in\GL_k\times G_{n-k}\isomorphic M_k$,
\begin{align}\label{eq:block-compatibility}
\sigma_{G_n}((a,g),(a',g'))=\sigma_{\GL_{k}}(a,a')\sigma_{G_{n-k}}(g,g')(\Upsilon(g),\det{a'})_2.
\end{align}
We also mention that $C_{\widetilde{G_n}}=\widetilde{C_{G_n}}$.

\subsection{Representations}\label{subsection:representations}
Let $G$ be an $l$-group (\cite{BZ1} 1.1). Representations of $G$ will be complex and smooth.
For a representation $\pi$ of $G$, $\pi^{\vee}$ is the representation contragradient to $\pi$. 
We say that $\pi$ is glued from representations $\pi_1,\ldots,\pi_k$, if $\pi$ has a filtration, whose quotients (which may be isomorphic or zero) are, after a permutation, $\pi_1,\ldots,\pi_k$.

Regular induction is denoted $\Ind$ and $\ind$ is the compact induction.
Induction is not normalized.

Let $\pi$ be as above and let $U<G$ be a unipotent subgroup, exhausted by its compact subgroups (always the case here). Let $\psi$ be a character of $U$. The Jacquet module of $\pi$ with respect to $U$ and $\psi$ is denoted $\pi_{U,\psi}$. It is a representation of the stabilizer of $\psi$ (and normalizer of $U$). The action is not normalized. We have an exact sequence
\begin{align*}
0\rightarrow \pi(U,\psi)\rightarrow \pi\rightarrow \pi_{U,\psi}\rightarrow0.
\end{align*}
The kernel $\pi(U,\psi)$ can be characterized by the Jacquet-Langlands lemma:
\begin{lemma}\label{lemma:Jacquet kernel as integral}(see e.g. \cite{BZ1} 2.33)
a vector $v$ in the space of $\pi$ belongs to $\pi(U,\psi)$ if and only if
\begin{align*}
\int_O\pi(u)v\ \psi^{-1}(u)\ du=0,
\end{align*}
for some compact subgroup $O<U$.
\end{lemma}
When $\psi=1$, we simply write $\pi(U)$ and $\pi_U$. 

Assume that $\widetilde{G}$ is a given $r$-th cover of $G$. Let $\varepsilon:\mu_r\rightarrow\C^*$ be a faithful character. A representation $\pi$ of $\widetilde{G}$ is called $\varepsilon$-genuine if it restricts to $\varepsilon$ on $\mu_r$. When $r=2$, such a representation is simply called genuine.

Let $\varphi:G\rightarrow\widetilde{G}$ be a section and assume $\pi$ and $\pi'$ are representations of $\widetilde{G}$, such that
$\pi$ is $\varepsilon$-genuine and $\pi'$ is $\varepsilon^{-1}$-genuine. Then $\pi\otimes \pi'$ (outer tensor product) is a representation of $G$ via $g\mapsto\pi(\varphi(g))\otimes\pi'(\varphi(g))$. The actual choice of $\varphi$ does not matter, whence
we omit it.


\subsection{Representations of Levi subgroups}\label{subsection:representations of Levi subgroups}
Levi subgroups of classical groups are direct products. The tensor of representations of the direct factors is usually 
used, to describe their representations. In passing to a cover group, these factors do not necessarily commute and then the tensor construction fails.

Except for the case of $k=n$, the preimages of $\GL_k$ and $G_{n-k}$ in $\widetilde{M}_k$ do not commute. The same phenomena occurs in $\GL_n$.
The following discussion describes a replacement for the usual tensor product. For more details see \cite{me8}.

The metaplectic tensor product in the context of $\GL_{n}$ has been studied by various authors \cite{FK,Su2,Kable,Mezo,Tk2}.
Our definitions were motivated by the construction of Kable \cite{Kable} (see Remark~\ref{remark:diff Kable tensor} below).

For any $H<G_n$, put $H^{\square}=\setof{h\in H}{\Upsilon(h)\in {F^{*2}}}$. The subgroup $H^{\square}$ is normal in $H$ and the quotient is a finite
abelian group. If $\xi$ is a representation of $\cover{H}$, let $\xi^{\square}=\xi|_{\cover{H}^{\square}}$. Assume $0<k<n$. According to \eqref{eq:block-compatibility}, the preimages of $\GL_k^{\square}$ and $G_{n-k}^{\square}$ are commuting in the cover. Then if $\rho$ and $\pi$ are genuine representations of $\widetilde{H}_1<\widetilde{\GL}_k$ and $\widetilde{H}_2<\widetilde{G}_{n-k}$, the representation $\rho^{\square}\otimes\pi^{\square}$ is a genuine representation of
\begin{align*}
p^{-1}(H_1^{\square}\times H_2^{\square})\isomorphic \lmodulo{\setof{(\epsilon,\epsilon)}{\epsilon\in\mu_2}}{(\widetilde{H}_1^{\square}\times\widetilde{H}_1^{\square})}.
\end{align*}
Put $H=H_1H_2$ and define
\begin{align*}
\mathcal{I}^{\square}(\rho,\pi)=\ind_{p^{-1}(H_1^{\square}\times H_2^{\square})}^{\cover{H}}{(\rho^{\square}\otimes\pi^{\square})}.
\end{align*}
When $k=n$, $\widetilde{\GL}_n$ and $\widetilde{G}_0$ are commuting, then the metaplectic tensor is defined as usual. 

\begin{remark}\label{remark:diff Kable tensor}
The arguments in \cite{Kable} do not readily extend to $G_n$, mainly because $C_{G_n}<G_n^{\square}$ for all $n$, and then $C_{G_n}$ does not play a role similar to that of $C_{\GL_n}$ in the cover.
\end{remark}

We will use Mackey Theory to relate this induced representation to $\rho$ and $\pi$. We reproduce the following result from \cite{me8}, which mimics \cite{Kable} (Theorem~3.1) in our context.
\begin{lemma}\label{lemma:induced representation composition factors}
The representation $\mathcal{I}^{\square}(\rho,\pi)$ is a direct sum of $[F^*:{F^{*2}}]$ copies of
\begin{align*}
\ind_{p^{-1}(H_1^{\square}\times H_2)}^{\widetilde{H}}(\rho^{\square}\otimes\pi).
\end{align*}
\end{lemma}
\begin{proof}[Proof of Lemma~\ref{lemma:induced representation composition factors}]
Since $p^{-1}(H_1^{\square}\times H_2^{\square})$ is normal of finite index in $\widetilde{H}$ and $p^{-1}(H_1^{\square}\times H_2)$ modulo $p^{-1}(H_1^{\square}\times H_2^{\square})$ is abelian,
\begin{align}\label{eq:claim tensor following Kable 1}
\mathcal{I}^{\square}(\rho,\pi)=\bigoplus_{a\in\lmodulo{H_1^{\square}}{H_1}}\ind_{p^{-1}(H_1^{\square}\times H_2)}^{\widetilde{H}}(\rho^{\square}\otimes\omega_a\pi).
\end{align}
Here $\omega_a(h)=(\Upsilon(h),\det{a})_2$ ranges over the finite set of characters of $H_2$, which are trivial
on $H_2^{\square}$. By \eqref{eq:block-compatibility} if $a_0\in\widetilde{\GL}_k$ and $h_0\in\widetilde{G}_{n-k}$, $\rconj{a_0^{-1}}h_0=(\Upsilon(h_0),\det{a_0})_2h_0$. Hence $\rho^{\square}\otimes\omega_a\pi=\rconj{a}(\rconj{a^{-1}}(\rho^{\square})\otimes\pi)$ and the result follows.
\end{proof}

\subsection{The Weil representation}\label{subsection:the Weil representation}
We introduce the Weil representation, which plays an important role in this work. Let $n=2k$ and $\lambda$ be the symplectic bilinear form on $F^{n}$ defined by
$\lambda(u,v)=u\left(\begin{smallmatrix}&J_{k}\\-J_{k}\end{smallmatrix}\right)\transpose{v}$, where $u$ and $v$ are regarded as rows.
Let $H_{n}$ be the $(\glnidx+1)$-dimensional Heisenberg group, with the group operation given by
\begin{align*}
(u_1,u_2;z_1)\cdot (v_1,v_2;z_2)=(u_1+v_1,u_2+v_2,z_1+z_2+\lambda((u_1,u_2),(v_1,v_2))),
\end{align*}
where $u_i,v_i\in F^k$ and $z_i\in F$.

Let $\Sp_k$ be the symplectic group defined with respect to $\lambda$, i.e., the group of
$g\in\GL_n$ such that $\lambda(ug,vg)=\lambda(u,v)$ for all $u,v\in F^n$. Let $\widetilde{\Sp}_k$ be the metaplectic double cover of $\Sp_k$, realized using the normalized Rao cocycle \cite{Rao}. The group $\Sp_k$ acts on $H_n$ via $g^{-1}(u_1,u_2;z)g=((u_1,u_2)g;z)$.

 Fix a nontrivial additive character $\psi$ of $F$. Let $\omega_{\psi}$ be the Weil representation of $H_{n}\rtimes\widetilde{\Sp}_k$, realized on the space $\mathcal{S}(F^{k})$ of Schwartz-Bruhat functions on $F^{k}$. Recall the following formulas for $\omega_{\psi}$ (see \cite{P}): for $\varphi\in\mathcal{S}(F^{k})$,
\begin{align}
&\omega_{\psi}((u_1,0;z))\varphi(x)=\psi(z)\varphi(x+u_1),\label{eq:Weil X action}\\
&\omega_{\psi}((0,u_2;0))\varphi(x)=\psi(x J_{k}\transpose{u_2})\varphi(x),\label{eq:Weil R action}\\
&\omega_{\psi}((\left(\begin{smallmatrix}I_{k}&u\\&I_{k}\end{smallmatrix}\right),\epsilon))\varphi(x)=\epsilon\psi(\half x J_{k}\transpose{u}\transpose{x})\varphi(x)\qquad (\epsilon\in\mu_2).\label{eq:Weil Y action}
\end{align}

Let $R=\{(0,u_2;0)\}<H_n$ and $U=\{\left(\begin{smallmatrix}I_{k}&u\\&I_{k}\end{smallmatrix}\right)\}<\Sp_k$. Since $U$ normalizes (in fact, commutes with) $R$, $(\omega_{\psi})_{R}$ is a $U$-module.
We will use the following simple observation.
\begin{claim}\label{claim:Jacquet of Weil 1 dim}
The vector spaces $(\omega_{\psi})_{R}$ and $(\omega_{\psi})_{RU}$ are one dimensional.
\end{claim}
\begin{proof}[Proof of Claim~\ref{claim:Jacquet of Weil 1 dim}]
According to Lemma~\ref{lemma:Jacquet kernel as integral} and \eqref{eq:Weil R action}, the space of $\omega_{\psi}(R)$ is $\mathcal{S}(F^k\setdiff 0)$.
Hence $(\omega_{\psi})_{R}$ is one dimensional. Then $(\omega_{\psi})_{RU}=((\omega_{\psi})_{R})_U$ is nonzero, because by Lemma~\ref{lemma:Jacquet kernel as integral} and \eqref{eq:Weil Y action}, a function $\varphi\in \mathcal{S}(F^k)$ such that
$\varphi(0)\ne0$ does not belong to $\omega_{\psi}(RU)$.
\end{proof}

We will also encounter the tensor $\omega_{\psi}\otimes\omega_{\psi^{-1}}$ of two Weil representations. This is a representation of $H_n$, trivial 
on $C_{H_n}$, and a representation of $\Sp_k$. Regarding it as a representation of $\lmodulo{C_{H_n}}{H_n}$, we can compute its twisted Jacquet modules. 
The group $\Sp_k$ acts transitively on the nontrivial characters of $\lmodulo{C_{H_n}}{H_n}$, hence we can consider only one nontrivial character.
\begin{claim}\label{claim:Jacquet modules along Hn of double Weil reps}
Let $\mu$ be a character of $H_n$, which is trivial on $C_{H_n}$. 
\begin{enumerate}
\item If $\mu=1$, $(\omega_{\psi}\otimes\omega_{\psi^{-1}})_{H_n,\mu}$ is the trivial one-dimensional representation of $\Sp_k$.
\item If $\mu(u_1,u_2;z)=\psi((u_1)_1)$, 
$(\omega_{\psi}\otimes\omega_{\psi^{-1}})_{H_n,\mu}$ is the trivial one-dimensional representation of 
$P_{n-1,1}^{\circ}\cap \Sp_k$.
\end{enumerate}
\end{claim} 
\begin{proof}[Proof of Claim~\ref{claim:Jacquet modules along Hn of double Weil reps}]
First assume $\mu=1$. 
Equality~\eqref{eq:Weil R action} implies that elements $\varphi\otimes\varphi'$ in the space of $\omega_{\psi}\otimes\omega_{\psi^{-1}}$, such that
the supports of $\varphi$ and $\varphi'$ (as functions in $\mathcal{S}(F^k)$) are different, vanish under the Jacquet module along $R$ (use Lemma~\ref{lemma:Jacquet kernel as integral}).
Hence the space of $(\omega_{\psi}\otimes\omega_{\psi^{-1}})_R$ is isomorphic to $\mathcal{S}(F^k)$. Since the action of $C_{H_n}$ is trivial
on $\omega_{\psi}\otimes\omega_{\psi^{-1}}$, $(\omega_{\psi}\otimes\omega_{\psi^{-1}})_R=(\omega_{\psi}\otimes\omega_{\psi^{-1}})_{RC_{H_n}}$.
Now $RC_{H_n}$ is a normal subgroup of $H_n$, whence $(\omega_{\psi}\otimes\omega_{\psi^{-1}})_{H_n}$ is a quotient of $(\omega_{\psi}\otimes\omega_{\psi^{-1}})_{RC_{H_n}}$.

It then follows
from \eqref{eq:Weil Y action} that the action of $U$ is trivial on
$(\omega_{\psi}\otimes\omega_{\psi^{-1}})_{RC_{H_n}}$ and in particular, on 
$(\omega_{\psi}\otimes\omega_{\psi^{-1}})_{H_n}$. The latter is a representation 
of $\Sp_k$, and because $\Sp_k$ is generated (as an abstract group) by the subgroups $\rconj{w}U$, $w\in\Sp_k$ (it is enough 
to take Weyl elements $w$), $(\omega_{\psi}\otimes\omega_{\psi^{-1}})_{H_n}$ must be a trivial representation of $\Sp_k$. 

Moreover
$(\omega_{\psi}\otimes\omega_{\psi^{-1}})_{H_n}$ is one-dimensional. This can be seen as follows. Replace
a function $f\in \mathcal{S}(F^k)=(\omega_{\psi}\otimes\omega_{\psi^{-1}})_{RC_{H_n}}$ with its Fourier transform
$\widehat{f}(x)=\int_{F^k}f(y)\psi(x(\transpose{y}))dy$. This changes the module structure, but the action of $\Sp_k$ remains trivial. The action of
$(u_1,0;0)\in H_n$ is now given by
\begin{align*}
(u_1,0;0)\cdot\widehat{f}(x)=\psi^{-1}(x\transpose{u_1})\widehat{f}(x)
\end{align*}
(instead of \eqref{eq:Weil X action}). Next apply Lemma~\ref{lemma:Jacquet kernel as integral}, the space of
$(\omega_{\psi}\otimes\omega_{\psi^{-1}})_{RC_{H_n}}(\lmodulo{(RC_{H_n})}{H_n})$ is $\mathcal{S}(F^k-0)$. 

We conclude that $(\omega_{\psi}\otimes\omega_{\psi^{-1}})_{H_n}$ is the trivial one-dimensional representation of $\Sp_k$. 

Now consider the case of the nontrivial $\mu$, given in the statement of the claim. Since $\mu|_{R}=1$, 
$(\omega_{\psi}\otimes\omega_{\psi^{-1}})_{H_n,\mu}$ is a quotient of 
$(\omega_{\psi}\otimes\omega_{\psi^{-1}})_{RC_{H_n}}$ and in particular, a trivial representation of $U$. In coordinates,
\begin{align*}
P_{n-1,1}^{\circ}\cap \Sp_k=\left\{\left(\begin{array}{ccc}1&u&v\\&g&*\\&&1\end{array}\right):g\in\Sp_{k-1}\right\}.
\end{align*}
(In particular it stabilizes $\mu$.) We see that $U<P_{n-1,1}^{\circ}\cap \Sp_k$ and using conjugations of $U$ by
elements from $\Sp_{k-1}$, it follows that $(\omega_{\psi}\otimes\omega_{\psi^{-1}})_{H_n,\mu}$ is the trivial representation of
$P_{n-1,1}^{\circ}\cap \Sp_k$. 

To show this is a one-dimensional representation, argue as above using the Fourier 
transform, the space of
$(\omega_{\psi}\otimes\omega_{\psi^{-1}})_{RC_{H_n}}(\lmodulo{(RC_{H_n})}{H_n,\mu})$ is $\mathcal{S}(F^k-(-1,0,\ldots,0))$. 
\end{proof}
\subsection{Exceptional or small representations}\label{subsection:The exceptional representations}
We describe the exceptional representations that appear in this work. Kazhdan and Patterson \cite{KP} introduced and studied these
representations for $\GL_n$. For $G_n$ these are essentially the small representations of Bump, Friedberg and Ginzburg \cite{BFG},
who developed them using a cover of $\SO_{2n+1}$ (see Section~\ref{section:relevance to SO(2n+1)}). Their construction was extended to $G_n$ in \cite{me8}.

Let $G$ be either $\GL_n$ or $G_n$. Let $B$ be the Borel subgroup of $G$, $T$ be the maximal torus and $\Delta$ be the subset of simple roots. For $\alpha\in\Delta$,
if $\alpha$ is a long root and $n>1$, put $\mathfrak{l}(\alpha)=2$, otherwise $\mathfrak{l}(\alpha)=1$. Denote by $\alpha^{\vee}$ the coroot of $\alpha$.

Let $\xi$ be a genuine character of $C_{\cover{T}}$. We say that $\xi$ is exceptional if
\begin{align*}
\xi(\mathfrak{s}(\alpha^{\vee}(x^{\mathfrak{l}(\alpha)})))=|x|,\qquad\forall \alpha\in\Delta,x\in F^*.
\end{align*}
The character $\xi$ corresponds to a genuine irreducible representation $\rho(\xi)$ of $\widetilde{T}$ (when $n>1$, $\widetilde{T}$ is a $2$-step nilpotent subgroup). Since $\xi$ is exceptional,
the representation $\Ind_{\widetilde{B}}^{\widetilde{G}}(\delta_{B}^{1/2}\rho(\xi))$ has a unique irreducible quotient $\theta$, called an
exceptional representation. Note that $\theta$ is admissible. Occasionally, we use the notation
$\theta^{G}$ to record the group.

We appeal to the following explicit description of $C_{\widetilde{T}}$:
$C_{\widetilde{T}_{\GL_n}}=\widetilde{T}_{\GL_n}^2C_{\widetilde{\GL}_n}$ with
$T_{\GL_n}^2=\setof{t^2}{t\in T_{\GL_n}}$,
$C_{\widetilde{\GL}_n}=\setof{z\cdot I_n}{z\in F^{*\mathe}}$, where
$\mathe$ is $1$ if $n$ is odd, otherwise $\mathe=2$; $C_{\widetilde{T}_{n+1}}=C_{\widetilde{T}_{\GL_n}}\widetilde{G}_0$, and if
$T_{n+1}^2=T_{\GL_n}^2G_0$, $C_{\widetilde{T}_{n+1}}=\widetilde{T}_{n+1}^2C_{\widetilde{\GL}_n}$ (see \cite{KP} p.~57 and \cite{me8} Section~2.1.6). 
Note that $G_0^{\square}=G_0$. Furthermore, the cocycle $\sigma_G$ satisfies $\sigma_G(z\cdot I_n,z'\cdot I_n)=(z,z')^{\lceil n/2\rceil}$.

The exceptional characters can be parameterized in the following way. Start by fixing a character $\xi_0$ of
$p(C_{\widetilde{T}})$, trivial on $p(C_{\widetilde{G}})$. This does not determine $\xi_0$ uniquely.
For $\GL_n$, take $\xi_0=\delta_{B_{\GL_n}}^{1/4}$. For $G_n$ take $\xi_0$ whose restriction to
$p(C_{\widetilde{T}_{\GL_n}})$ is $\delta_{B_{\GL_n}}^{1/4}\cdot|\det{}|^{(n+1)/4}$.
Given that $\xi_0$ is trivial on $p(C_{\widetilde{G}_n})$, this determines $\xi_0$ uniquely for $G_n$ (see \cite{me8} Section~2.3.3 for
an explicit formula).

Now for any given $\xi$, there is a character $\chi$ of $F^*$ (called ``determinantal character" in \cite{BG}) and a nontrivial additive
character $\psi$ of $F$ such that 
\begin{align}\label{eq:explicit construction exceptional characters}
\xi(\epsilon\mathfrak{s}(td))=\epsilon\xi_0(td)\chi(\Upsilon(td))\gamma_{\psi}^{\lceil n/2\rceil}(z),\qquad
\forall t\in T^2,d=z^{\mathe}\cdot I_n\in T,\epsilon\in\mu_2.
\end{align}
The corresponding exceptional
representation will be denoted $\theta_{\chi,\psi}$.

\begin{claim}\label{claim:taking characters out of exceptional}
We have $\theta_{\chi,\psi}=\chi\theta_{1,\psi}$, where on the right-hand side we
pull back $\chi$ to a non-genuine character of $\widetilde{G}$ via $g\mapsto\chi(\Upsilon(p(g)))$.
Additionally, if $\psi_0$ is another additive character of $F$,
$\theta_{\chi,\psi}=\eta\theta_{\chi,\psi_0}$ for some square trivial character $\eta$ of $F^*$.
\end{claim}
\begin{proof}[Proof of Claim~\ref{claim:taking characters out of exceptional}]
The first assertion is clear. Write $\psi(x)=\psi_0(\alpha x)$ for some
$\alpha\in F^*$. Then $\gamma_{\psi}=\eta_0\gamma_{\psi_0}$, where $\eta_0(z)=(\alpha,z)_2$. Since $\eta_0$ is trivial
on $F^{*2}$ and
\begin{align*}
\eta_0(\Upsilon(d))=\eta_0(\det{d})=\eta_0(z^n)=\eta_0(z)
\end{align*}
(the last equality is trivial if $n$ is even, because then $z\in F^{*2}$), we obtain
\begin{align*}
\chi(\Upsilon(td))\gamma_{\psi}^{\lceil n/2\rceil}(z)=
(\eta_0^{\lceil n/2\rceil}\chi)(\Upsilon(td))\gamma_{\psi_0}^{\lceil n/2\rceil}(z).
\end{align*}
Hence $\theta_{\chi,\psi}=\theta_{\eta_0^{\lceil n/2\rceil}\chi,\psi_0}=\eta_0^{\lceil n/2\rceil}\theta_{\chi,\psi_0}$.
\end{proof}


Exceptional representations have a useful inductive property. 
Let $\theta$ be an exceptional representation of $\widetilde{G}_{n}$. Following the arguments of Bump, Friedberg and Ginzburg \cite{BFG} (Theorem~2.3),
we computed $\theta_{U_k}$ (\cite{me8}). For $0<k<n$,
\begin{align}\label{eq:containment of Jacquet module along maximal unipotent}
(\theta_{\chi,\psi})_{U_k}\subset\mathcal{I}^{\square}(\theta^{\GL_k}_{|\cdot|^{(2n-k-1)/4}\chi,\psi},\theta^{G_{n-k}}_{\chi,\psi}).
\end{align}
If $k$ (resp. $n-k$) is odd, the exceptional representation of $\widetilde{\GL}_k$ (resp. $\widetilde{G}_{n-k}$) is unique only up
to varying the character $\psi$, or multiplying $\chi$ by a square trivial character of $F^*$. In any case, the space
on the \rhs\ of \eqref{eq:containment of Jacquet module along maximal unipotent} is unique, because by Claim~\ref{claim:taking characters out of exceptional},
the exceptional representation obtained by such a change to $\psi$ or $\chi$, has the same restriction to $\widetilde{\GL}_k^{\square}$ (resp. $\widetilde{G}_{n-k}^{\square}$)
as the original representation.

If $k=n$,
\begin{align}\label{eq:containment of Jacquet module along maximal unipotent 2}
(\theta_{\chi,\psi})_{U_n}=\theta^{\GL_n}_{|\cdot|^{(n-1)/4}\chi,\psi}\otimes\theta^{G_0}_{\chi,1}.
\end{align}
In \cite{me8} we did not find the precise exceptional representations appearing on the \rhs\ of \eqref{eq:containment of Jacquet module along maximal unipotent 2}, but this is
simple to obtain:
since $C_{\widetilde{T}_{n+1}}=C_{\widetilde{T}_{\GL_n}}\widetilde{G}_0$, there is an exceptional character $\xi_1$ of $C_{\widetilde{T}_{\GL_n}}$
such that $\rconj{w_0}\xi=\rconj{w_0'}\xi_1\otimes\xi|_{\widetilde{G}_0}$,
where $w_0$ and $w_0'$ are the longest Weyl elements in the Weyl groups of $G_n$ and $\GL_n$ (see \cite{me8} Claim~2.18 for details). It remains to write $\xi_1$
using \eqref{eq:explicit construction exceptional characters}.

\begin{remark}
The reason for the imprecise result when $k<n$ is the
lack of a definition for a tensor product. These results are sufficient for our applications.
\end{remark}

For $\GL_n$, Kable proved a result similar to \eqref{eq:containment of Jacquet module along maximal unipotent 2} for all standard unipotent radicals,
with the tensor replaced by his metaplectic tensor (\cite{Kable} Theorem~5.1 (4)).

Exceptional representations enjoy the vanishing of a large class of twisted Jacquet modules.
The following result is the extension of Theorem~2.6 of \cite{BFG} and Proposition~3 of \cite{BFG2} to $G_n$ (this extension
appeared in \cite{me8}). It was used in \cite{BFG,BFG2} (for $\SO_{2n+1}$) to deduce all vanishing properties.

The unipotent radical $U_1$ is abelian. A character $\psi^{(1)}$ of $U_1$ takes the form
\begin{align*}
\psi^{(1)}(\left(\begin{array}{ccc}1&u&*\\&I_{2n-1}&*\\&&1\end{array}\right))=\psi(ua),
\end{align*}
where $a\in F^{2n-1}$ is a column. The length of $a$ is defined to be $\transpose{a} J_{2n-1}a$.
When $\psi$ is fixed and clear from the context, we also refer to $\transpose{a}J_{2n-1}a$ as the length of $\psi^{(1)}$.
\begin{theorem}\label{theorem:main vanishing Jacquet of small rep}
If the length of $\psi^{(1)}$ is nonzero, $\theta_{U_1,\psi^{(1)}}=0$.
\end{theorem}
See \cite{me8} (Lemma~2.25) for the details.
\section{Twisted Jacquet modules of $\theta^{G_n}$}\label{section:Heisenberg jacquet modules}

In this section we describe the twisted Jacquet modules of $\theta=\theta^{G_n}$ with respect to the center of $U_n$. These modules appear in a filtration of $\theta$ as a $\widetilde{Q}_n$-module and will be used in Section~\ref{section:Distinguished representations}.

The group $\GL_n$, embedded in $M_n$, acts on the set of characters of $C=C_{U_n}$ with finitely many orbits. Let $\psi$ be a nontrivial
additive character of $F$. For any $0\leq k\leq \lfloor n/2 \rfloor$, define a character of $C$ by
\begin{align*}
\psi_k(c)=\psi(\sum_{i=1}^kc_{n-2i+1,n+2i})
\end{align*}
($c$ is regarded as a $(2n+1)\times(2n+1)$ unipotent matrix in $\SO_{2n+1}$).
The stabilizer of $\psi_k$ in $\widetilde{M}_n$ is $\widetilde{\mathrm{St}}_{\psi_k}$, with
\begin{align*}
\mathrm{St}_{\psi_k}=\left\{\left(\begin{array}{cc}a&z\\0&b\end{array}\right):a\in\GL_{n-2k},b\in \Sp_{k}\right\}\times G_0.
\end{align*}
Here $\Sp_{k}$ is the symplectic group in $2k$ variables, corresponding to a symplectic bilinear form defined according to $\psi_{k}$.
Then $\theta_{C,\psi_k}$ is a representation of $\widetilde{\mathrm{St}}_{\psi_k}\ltimes U_n$.

Regarding 
$\GL_{n-2k}$ and $\Sp_k$ as subgroups of $\mathrm{St}_{\psi_k}$,
\begin{align*}
\mathrm{St}_{\psi_k}=((\GL_{n-2k}\times \Sp_k)\ltimes Z_{n-2k,2k}) \times G_0.
\end{align*}
($Z_{n-2k,2k}$ was given in Section~\ref{subsection:GL_n and its cover}.)

We turn to the proof of Theorem~\ref{theorem:small rep is weakly minimal}.
Namely, for $n=2k$, $\theta_{C,\psi_{k}}$ is the direct sum of copies of the Weil representation $\omega_{\psi}$.
Here is an outline of the proof. 
The theory of smooth representations of Jacobi groups (\cite{Dijk,MVW,BSch}) implies that any such representation, with a central character $\psi$, takes the form
$\kappa\otimes\omega_{\psi}$, where the Heisenberg group acts trivially on the space of $\kappa$, and the action of the
symplectic group separates into an action on the space of $\kappa$, and one on the space of $\omega_{\psi}$. The vanishing properties of $\theta$ - Theorem~\ref{theorem:main vanishing Jacquet of small rep}, will show that
$\kappa$ is trivial.

\begin{proof}[Proof of Theorem~\ref{theorem:small rep is weakly minimal}]
Put $k=n/2$. The image of $G_0$ in $G_n$ is $C_{G_n}$. Hence $\mathrm{St}_{\psi_{k}}$ is the direct product of a Jacobi group and
$G_0$. Moreover, $\widetilde{G}_0=C_{\widetilde{G_n}}$ (see Section~\ref{subsection:metaplectic GSpin}) and because
$\theta$ is an irreducible representation, $\widetilde{G}_0$ acts by the central character of $\theta$. Therefore in the proof we ignore the $G_0$ part of
$\mathrm{St}_{\psi_k}$.

We may replace $\psi_k$ by any character of $C$ in the same $\GL_n$-orbit, since the Jacquet module will be isomorphic. For convenience,
we redefine $\psi_{k}(c)=\psi(\sum_{i=1}^kc_{i,{n}+1+i})$. We use the notation of Section~\ref{subsection:the Weil representation}. The stabilizer $\mathrm{St}_{\psi_k}$ is now the symplectic group defined with respect to the form $\lambda$. The cover $\widetilde{\mathrm{St}}_{\psi_k}$ is a nontrivial double cover.
We have an epimorphism $\ell:\widetilde{\mathrm{St}}_{\psi_k}\ltimes U_n\rightarrow \widetilde{\Sp}_k\ltimes H_n$:
\begin{align*}
&\ell(\left(\begin{array}{ccc}I_{{n}}&u&z\\&1&-\transpose{u}J_{{n}}\\&&I_{{n}}\end{array}\right))=(\transpose{u}J_{{n}};\half(\sum_{i=1}^kz_{i,i}-
\sum_{i=k+1}^{{n}}z_{i,i}))\in H_n,\\
&\ell((\left(\begin{array}{ccc}g&&\\&1&\\&&J_{{n}}\transpose{g}^{-1}J_{{n}}\end{array}\right),\epsilon))=(J_{{n}}\transpose{g}^{-1}J_{{n}},\epsilon)\in\widetilde{\Sp}_k \qquad(\text{$g$ preserves $\lambda$},\quad\epsilon\in\mu_2).
\end{align*}
The kernel of $\ell$ is contained in the kernel of $\psi_k$ \footnote{In \cite{me7} p.~25 it was incorrectly stated they are equal.}. Therefore we can regard $\theta_{C,\psi_k}$ as a genuine representation
of $\widetilde{\Sp}_k\ltimes H_n$. As such, it is isomorphic to $\kappa\otimes\omega_{\psi}$ (see e.g. \cite{BSch} p.~28), where
$\kappa$ is a non-genuine representation on a space $\mathbb{V}$, and the action is
given by
\begin{align*}
(g,\epsilon)h\cdot(\mathfrak{f}\otimes\varphi)=\kappa(g)\mathfrak{f}\otimes \omega_{\psi}((g,\epsilon)h)\varphi,\qquad g\in \Sp_k,\quad h\in H_n,\quad \mathfrak{f}\in\mathbb{V},\quad\varphi\in\mathcal{S}(F^k).
\end{align*}
We must show that $\kappa$ is trivial. Consider the subgroup
\begin{align*}
Y=\left\{\left(\begin{array}{ccc}1&0&y\\&I_{2(n-1)}&0\\&&1\end{array}\right)\right\}<\Sp_k.
\end{align*}
Since $\Sp_k$ is generated 
by the conjugates $\rconj{x}Y$ where $x\in\Sp_k$, it is enough to prove
invariance under $Y$. That is, we show
\begin{align}\label{eq:invariancy under Y to prove}
\kappa_{Y}=\kappa.
\end{align}

A character of $Y$ takes the form $y\mapsto\psi(\alpha y)$ for some $\alpha\in F$ (on the \lhs\ $y$ is regarded as a matrix, on the \rhs\ as an element of $F$). The action of the torus of $\Sp_k$ on
the nontrivial characters of $Y$ has finitely many orbits, namely the different square classes in $F^*$. Each of these orbits is open. Therefore the kernel of the Jacquet functor $\kappa_{Y}$ is filtered by representations induced from $\kappa_{Y,\psi(\alpha\cdot)}$, where $\alpha$
ranges over the square classes (\cite{BZ1} 5.9-5.12). Hence \eqref{eq:invariancy under Y to prove} follows if we prove $\kappa_{Y,\psi(\alpha\cdot)}=0$ for any $\alpha\ne0$.

Consider the subgroup $X=Y\cdot\setof{(0,u_2;0)\in H_n}{u_2=(0,\ldots,0,r)}$ (a direct product). The
epimorphism $\ell$ splits over $X$, hence
there is a subgroup $U<\widetilde{\mathrm{St}}_{\psi_k}U_n$ isomorphic to $X$. In fact,
\begin{align*}
U=\left\{\left(\begin{array}{ccccccc}1&0&y&r&0&0&-r^2/2\\&I_{n-2}&&&&&0\\&&1&&&&0\\&&&1&&&-r\\&&&&1&&-y\\&&&&&I_{n-2}&0\\&&&&&&1\end{array}\right)\right\}<U_1.
\end{align*}
The pullback of $\psi(\alpha\cdot)$ to $U$ is $\psi^{\star}(u)=\psi(-\alpha u_{1,n})$. Observe that $(\theta_{C,\psi_k})_{U,\psi^{\star}}=0$. Indeed, $(\theta_{C,\psi_k})_{U,\psi^{\star}}=\theta_{CU,\psi_k\psi^{\star}}$, which is a quotient of $\theta_{U(C\cap U_1),\psi_k\psi^{\star}}$. Since for $u\in U(C\cap U_1)$, $\psi_k\psi^{\star}(u)=\psi(-\alpha u_{1,n}+u_{1,n+2})$, any extension of $\psi_k\psi^{\star}$ to a character of $U_1$ is a character of nonzero length. Thus Theorem~\ref{theorem:main vanishing Jacquet of small rep} yields $\theta_{U(C\cap U_1),\psi_k\psi^{\star}}=0$ whence $(\theta_{C,\psi_k})_{U,\psi^{\star}}=0$. Therefore by Lemma~\ref{lemma:Jacquet kernel as integral}, for any $\mathfrak{f}\otimes\varphi$ there is a compact $\mathcal{O}<Y\cdot R$ ($R=\{(0,u_2;0)\}$, see Section~\ref{subsection:the Weil representation}) such that
\begin{align}\label{eq:Y acts trivially relation to twisted Jacquet vanishing of theta}
\int_{\mathcal{O}}yr\cdot(\mathfrak{f}\otimes \varphi)\ \psi^{-1}(\alpha y)\ dr\ dy=0.
\end{align}

According to Claim~\ref{claim:Jacquet of Weil 1 dim} and Lemma~\ref{lemma:Jacquet kernel as integral} there is
$\varphi\in\mathcal{S}(F^k)$ such that for all compact subgroups $\mathcal{Y}<Y$ and $\mathcal{R}<R$,
\begin{align}\label{eq:def of varphi to take for tensor vanishing}
\varphi^{\mathcal{Y},\mathcal{R}}=\int_{\mathcal{Y}}\int_{\mathcal{R}}\omega_{\psi}(yr)\varphi\ dy\ dr\ne0.
\end{align}
Take $\mathfrak{f}\in\mathbb{V}$. We will show that for large enough $\mathcal{Y}$ and $\mathcal{R}$,
 \begin{align}\label{eq:tensor vanishes to show}
 \int_{\mathcal{Y}}\kappa(y_1)\mathfrak{f}\ \psi^{-1}(\alpha y_1)\ dy_1\ \otimes\ \varphi^{\mathcal{Y},\mathcal{R}}=0.
 \end{align}
 This along with \eqref{eq:def of varphi to take for tensor vanishing} imply that $\mathfrak{f}$ belongs to the space of $\kappa(Y,\psi(\alpha\cdot))$.

 Plugging \eqref{eq:def of varphi to take for tensor vanishing} into \eqref{eq:tensor vanishes to show} and changing variables leads to
\begin{align*}
\int_{\mathcal{Y}}\left(\int_{\mathcal{Y}}\int_{\mathcal{R}} \kappa(y)\mathfrak{f}\ \otimes\ \omega_{\psi}(yry_1^{-1})\varphi\ \psi^{-1}(\alpha y)\ dr\ dy\right)\ dy_1.
\end{align*}
We will show that the inner $drdy$-integration vanishes for all $y_1\in\mathcal{Y}$. Fix $y_1$. Since $\kappa|_{H_n}$ is trivial,
this inner integration equals
\begin{align}\label{eq:part 00}
\int_{\mathcal{Y}}\int_{\mathcal{R}}yr\cdot(\mathfrak{f}\otimes\omega_{\psi}(y_1^{-1})\varphi)\ \psi^{-1}(\alpha y)\ dr\ dy.
\end{align}
Again resorting to Claim~\ref{claim:Jacquet of Weil 1 dim},
\begin{align}\label{eq:part 1}
\omega_{\psi}(y_1^{-1})\varphi=c_{y_1}\varphi+\varphi_{y_1}^{\circ},
\end{align}
where $c_{y_1}\in\C$ and $\varphi_{y_1}^{\circ}$ belongs to the space of $\omega_{\psi}(R)$. Since $y_1$ varies in a compact subgroup,
there is a large enough $\mathcal{R}$ such that
\begin{align}\label{eq:part 2}
\int_{\mathcal{R}}\omega_{\psi}(r)\varphi_{y_1}^{\circ}\ dr=0,\qquad\forall y_1\in\mathcal{Y}.
\end{align}
Furthermore \eqref{eq:Y acts trivially relation to twisted Jacquet vanishing of theta} implies that for large
$\mathcal{Y}$ and $\mathcal{R}$,
\begin{align*}
\int_{\mathcal{Y}}\int_{\mathcal{R}} yr\cdot(\mathfrak{f}\otimes \varphi)\ \psi^{-1}(\alpha y)\ dr\ dy=0
\end{align*}
and then for any $c\in\C$,
\begin{align}\label{eq:part 3}
\int_{\mathcal{Y}}\int_{\mathcal{R}} yr\cdot(\mathfrak{f}\otimes c\varphi)\ \psi^{-1}(\alpha y)\ dr\ dy=0.
\end{align}
Combining \eqref{eq:part 1}-\eqref{eq:part 3} we conclude that for sufficiently large $\mathcal{R}$ and $\mathcal{Y}$,
the inner $drdy$-integration \eqref{eq:part 00} vanishes. Note the order of selecting the compact subgroups: first, choose $\mathcal{R}$ and $\mathcal{Y}$
which ensure \eqref{eq:part 3}, they depend only on $\mathfrak{f}$ and $\varphi$. Then, increase $\mathcal{R}$ to have \eqref{eq:part 2}, it will depend on $\varphi$ and $\mathcal{Y}$. This completes the proof of \eqref{eq:tensor vanishes to show} and thereby \eqref{eq:invariancy under Y to prove}. We conclude that $\kappa$ is trivial.
\end{proof}

In the more general case, for arbitrary $k$, we are less precise.
\begin{proposition}\label{proposition:Jacquet module C and k}
There are exceptional representations $\theta^{\GL_{n-2k}}$ and $\theta^{G_{2k}}$ such that
$\theta_{C,\psi_k}$ is embedded in a finite direct sum of copies of the representation
\begin{align*}
\vartheta\otimes(\theta^{G_{2k}})_{C_{U_{2k}},\psi_k},\qquad
\vartheta=\begin{dcases}\theta^{\GL_{n}}&k=0,\\
\ind_{\widetilde{\GL}_{n-2k}^{\square}}^{\widetilde{\GL}_{n-2k}}((\theta^{\GL_{n-2k}})^{\square})&k>0.
\end{dcases}
\end{align*}
Here $\vartheta\otimes(\theta^{G_{2k}})_{C_{U_{2k}},\psi_k}$ is
regarded as a representation of $\widetilde{\mathrm{St}}_{\psi_k}\ltimes U_n$ by extending it trivially on $U_{n-2k}$.
The representations $\theta^{\GL_{n-2k}}$ and $\theta^{G_{2k}}$ are related to
$\theta$ via \eqref{eq:containment of Jacquet module along maximal unipotent} and \eqref{eq:containment of Jacquet module along maximal unipotent 2}.
If $k=0$, the
embedding is in fact an isomorphism and there is only one summand, otherwise there are $[F^*:F^{*2}]$ summands.
\end{proposition}
\begin{proof}[Proof of Proposition~\ref{proposition:Jacquet module C and k}]
For $n=1$, we have $C=U_1$ and $k=0$, whence $\theta_{C,\psi_k}=\theta_{U_1}$ and the result follows immediately
from \eqref{eq:containment of Jacquet module along maximal unipotent 2}.

Assume $n>1$. Further assume $k<n/2$, otherwise there is nothing to prove.
The main part of the proof is to show that $U_{n-2k}$ acts trivially on $\theta_{C,\psi_k}$. Of course, this holds for
$U_{n-2k}\cap C$. Let $V_k=U_{n-2k}\cap U_n$ and note that $Z_{n-2k,2k}=U_{n-2k}\cap M_n$. Clearly $U_{n-2k}=V_k\rtimes Z_{n-2k,2k}$.
The following claims imply that the action of $U_{n-2k}$ is trivial.
\begin{claim}\label{claim:V_k acts trivially}$\theta_{C,\psi_k}=\theta_{V_kC,\psi_{k}}$.
\end{claim}
\begin{claim}\label{claim:Z_n-2k,k acts trivially}$\theta_{V_kC,\psi_{k}}=\theta_{U_{n-2k}C,\psi_{k}}$.
\end{claim}
Before proving the claims, let us deduce the proposition. Clearly
$\theta_{U_{n-2k}C,\psi_{k}}=(\theta_{U_{n-2k}})_{C_{U_{2k}},\psi_{k}}$. Assume $k>0$. Then
by \eqref{eq:containment of Jacquet module along maximal unipotent},
\begin{align}\label{eq:equality to correct if k = 0}
\theta_{U_{n-2k}C,\psi_{k}}\subset\mathcal{I}^{\square}(\theta^{\GL_{n-2k}},\theta^{G_{2k}})_{C_{U_{2k}},\psi_k}.
\end{align}
According to Lemma~\ref{lemma:induced representation composition factors}, $\mathcal{I}^{\square}(\theta^{\GL_{n-2k}},\theta^{G_{2k}})$ is the finite direct sum of $[F^*:F^{*2}]$ copies of
\begin{align*}
\ind_{p^{-1}(\GL_{n-2k}^{\square}\times G_{2k})}^{\widetilde{M}_{n-2k}}((\theta^{\GL_{n-2k}})^{\square}\otimes\theta^{G_{2k}}).
\end{align*}
Let $\mathrm{St}'_{\psi_k}$ be the stabilizer of $\psi_k$ in $M_{2k}$, when $\psi_k$ is regarded as a character of $C_{U_{2k}}$ and $M_{2k}<G_{2k}<G_n$.
The double coset space
\begin{align*}
\rmodulo{\lmodulo{\GL_{n-2k}^{\square}\times G_{2k}}{M_{n-2k}}}{\GL_{n-2k}\times\mathrm{St}'_{\psi_k}}
\end{align*}
is trivial. Then by virtue of the Geometric Lemma of Bernstein and Zelevinsky \cite{BZ2} (Theorem~5.2),
\begin{align*}
\ind_{p^{-1}(\GL_{n-2k}^{\square}\times G_{2k})}^{\widetilde{M}_{n-2k}}((\theta^{\GL_{n-2k}})^{\square}\otimes\theta^{G_{2k}})_{C_{U_{2k}},\psi_k}
=\ind_{p^{-1}(\GL_{n-2k}^{\square}\times \mathrm{St}'_{\psi_k})}^{p^{-1}(\GL_{n-2k}\times \mathrm{St}'_{\psi_k})}((\theta^{\GL_{n-2k}})^{\square}\otimes(\theta^{G_{2k}})_{C_{U_{2k}},\psi_k}).
\end{align*}
Equality~\eqref{eq:block-compatibility} implies the subgroups $\widetilde{\GL}_{n-2k}$ and $\widetilde{\Sp}_k$ commute. Therefore
\begin{align*}
\ind_{p^{-1}(\GL_{n-2k}^{\square}\times \mathrm{St}'_{\psi_k})}^{p^{-1}(\GL_{n-2k}\times \mathrm{St}'_{\psi_k})}((\theta^{\GL_{n-2k}})^{\square}\otimes(\theta^{G_{2k}})_{C_{U_{2k}},\psi_k})
=(\ind_{\widetilde{\GL}_{n-2k}^{\square}}^{\widetilde{\GL}_{n-2k}}(\theta^{\GL_{n-2k}})^{\square})\otimes(\theta^{G_{2k}})_{C_{U_{2k}},\psi_k}.
\end{align*}
Thus $\mathcal{I}^{\square}(\theta^{\GL_{n-2k}},\theta^{G_{2k}})_{C_{U_{2k}},\psi_k}$ is the direct sum of representations $\vartheta\otimes(\theta^{G_{2k}})_{C_{U_{2k}},\psi_k}$.
The proposition follows from this. Note that for $k=0$,
$\theta_{U_{n-2k}C,\psi_{k}}=\theta_{U_{n}}$ and we can apply \eqref{eq:containment of Jacquet module along maximal unipotent 2} instead of
\eqref{eq:containment of Jacquet module along maximal unipotent}, then
\eqref{eq:equality to correct if k = 0} becomes $\theta_{U_{n}C}=\theta^{\GL_{n}}\otimes\theta^{G_{0}}$.
%
\begin{proof}[Proof of Claim~\ref{claim:V_k acts trivially}]
A character $\psi_k^{\star}$ of $V_kC$ extending $\psi_k$ is defined by its restriction to the nontrivial coordinates on the $(n+1)$-th column of $v\in V_k$. We call $\psi_k^{\star}$ nontrivial if this restriction is nontrivial.
We prove that the Jacquet module of $\theta_{C,\psi_k}$ with respect to $V_k$ and $\psi^{\star}$ vanishes for any nontrivial $\psi^{\star}$. The group $\GL_{n-2k}<\mathrm{St}_{\psi_k}$ acts transitively on these characters. Therefore, it is enough to show
\begin{align*}
(\theta_{C,\psi_k})_{V_k,\psi_{k}^{\star}}=\theta_{V_kC,\psi_{k}^{\star}}=0,\qquad \psi_k^{\star}(v)=\psi(v_{1,n+1}), \quad v\in V_k.
\end{align*}
This follows immediately from Theorem~\ref{theorem:main vanishing Jacquet of small rep}, because any character of $U_1$ extending $\psi_{k}^{\star}|_{U_1}$ has a nonzero length. Thus $\theta_{C,\psi_k}=\theta_{V_kC,\psi_{k}}$.
\end{proof}

\begin{proof}[Proof of Claim~\ref{claim:Z_n-2k,k acts trivially}]
For $k=0$ there is nothing to prove ($V_0=U_n)$. Assume $k>0$.
The claim follows once we show that for any nontrivial character $\mu$ of $Z_{n-2k,2k}$,
\begin{align}\label{eq:lemma first vanishing k < n/2 action of Z}
(\theta_{V_kC,\psi_k})_{Z_{n-2k,2k},\mu}=0.
\end{align}
The group
$Z_{n-2k,2k}$ is abelian and $\GL_{n-2k}\times \Sp_k$ acts on the characters of $Z_{n-2k,2k}$. Write an element $z\in Z_{n-2k,2k}$ in the form
\begin{align*}
z=z(z_1,z_2,z_3,z_4)=\left(\begin{array}{cccc}I_{n-2k-1}&&z_1&z_2\\&1&z_3&z_4\\&&1\\&&&I_{2k-1}\end{array}\right).
\end{align*}
We may assume that $\mu$ does not depend on the coordinates of $z_1$ and $z_4$, and depends on $z_3$. For simplicity, also assume
$\mu(z(0,0,z_3,0))=\psi(z_3)$.

We use the local analog of ``exchanging roots", proved by Ginzburg, Rallis and Soudry \cite{GRS5} (Lemma~2.2). (For the global setting see \cite{G,GRS3,Soudry5,RGS}.) Let $Z_{1}<Z_{n-2k,2k}$ be the subgroup of elements $z(z_1,0,0,0)$ and $Z_{2,3,4}<Z_{n-2k,2k}$ be the subgroup
consisting of elements $z(0,z_2,z_3,z_4)$. Clearly $Z_{n-2k,2k}=Z_{1}\cdot Z_{2,3,4}$ (a direct product). Also consider the subgroup
\begin{align*}
E=\left\{\left(\begin{array}{ccc}I_{n-2k-1}\\e&1\\&&I_{2k}\end{array}\right)\right\}.
\end{align*}
In general, if $\pi$ is a smooth representation of $Z_{n-2k,2k}\rtimes E$, then by \cite{GRS5} (Lemma~2.2), as $Z_{2,3,4}$-modules
\begin{align*}
\pi_{Z_{n-2k,2k},\mu}=\pi_{Z_{2,3,4}\rtimes E,\mu}.
\end{align*}
Indeed, it is simple to check that the list of properties stated in the lemma are satisfied in this setting (in the notation
of \cite{GRS5}, $C=Z_{2,3,4}$, $X=Z_1$ and $Y=E$).
\begin{remark}
Lemma~2.2 of \cite{GRS5} was stated for unipotent subgroups of symplectic groups, but the arguments are general and hold in our setting.
See also Section~2.3 of \cite{GRS5}.
\end{remark}
It follows that as $Z_{2,3,4}$-modules
\begin{align*}
(\theta_{V_kC,\psi_k})_{Z_{n-2k,2k},\mu}=\theta_{(V_kC)\rtimes(Z_{2,3,4}E),\psi_k\mu}.
\end{align*}
Conjugating the \rhs\ by a Weyl element of $G_n$, whose
action on $N_n$ is given by the action of
\begin{align*}
\mathrm{diag}(\left(\begin{array}{cc}&I_{2k+1}\\I_{n-2k-1}\end{array}\right),1,
\left(\begin{array}{cc}&I_{n-2k-1}\\I_{2k+1}\end{array}\right)),
\end{align*}
we obtain that $\theta_{(V_kC)\rtimes(Z_{2,3,4}E),\psi_k\mu}$ is a quotient of
\begin{align*}
(\theta_{U_1,\psi_1})_{U_2',\psi_2}.
\end{align*}
Here $\psi_1(u)=\psi(u_{1,2})$, $U_2'$ is a certain subgroup of $U_2$ (obtained from the conjugation of $C$) and $\psi_2(u)=\psi(u_{2,2n-1})$. Note that $\psi_1$ corresponds
to $\mu$ and the coordinate $z_3$ while $\psi_2$ corresponds to the
character $\psi_k$ of $C$ and the $(n-2k+1,n+2k)$-th coordinate of $c\in C$. The character $\psi_2$ is nontrivial on $U_2'$. Finally, by Proposition~4 of Bump, Friedberg and Ginzburg \cite{BFG2} (which is easily extended to $G_n$, given the analog of Theorem~\ref{theorem:main vanishing Jacquet of small rep} in \cite{me8}),
$\theta_{U_1,\psi_1}$ is a quotient of $\theta_{U_2}$ (this is valid for $n\geq 3$, here $0<k<n/2$ whence $n\geq3$). Hence, $U_2$ acts trivially on
$\theta_{U_1,\psi_1}$ and therefore $(\theta_{U_1,\psi_1})_{U_2',\psi_2}=0$ and \eqref{eq:lemma first vanishing k < n/2 action of Z} follows.
\end{proof}
\end{proof}

\section{Distinguished representations}\label{section:Distinguished representations}
Let $G$ be either $\GL_n$ or $G_n$. Let $\tau$ be an admissible representation of $G$ with a central character $\omega_{\tau}$.
Assume that $\theta$ and $\theta'$ are a pair of exceptional representations of $\widetilde{G}$.
We say that $\tau$ is $(\theta,\theta')$-distinguished if
\begin{align*}
\Hom_{G}(\theta\otimes\theta',\tau^{\vee})\ne0.
\end{align*}
The following result describes the upper hereditary property of a distinguished representation of $\GL_n$, when induced to a representation of
$G_n$. Following the notation of Section~\ref{subsection:The exceptional representations},
we denote the exceptional representation of $\widetilde{G}$ corresponding to $\chi$ and $\psi$ by
$\theta^G_{\chi,\psi}$. 
For any representation $\sigma$ of $\GL_n$, $s\in\C$ and a character $\mu$ of $F^*$, one forms a representation $\sigma|\det{}|^s\otimes\mu$ of $M_n$. Put
\begin{align*}
\mathrm{I}(\sigma,s,\mu)=\Ind_{Q_n}^{G_n}(\delta_{Q_n}^{1/2}\sigma|\det{}|^s\otimes\mu).
\end{align*}
\begin{proposition}\label{proposition:upper heredity of dist 1}
Let $\tau$ be a $(\theta^{\GL_n}_{\chi,\psi},\theta^{\GL_n}_{\chi',\psi'})$-distinguished representation of $\GL_n$ and set
$\eta=(\chi\chi')^{-1}$. Then $\mathrm{I}(\tau,1/2,\eta)$ is a
$(\theta^{G_n}_{\chi,\psi},\theta^{G_n}_{\chi',\psi'})$-distinguished representation of $G_n$.
\end{proposition}
\begin{proof}[Proof of Proposition~\ref{proposition:upper heredity of dist 1}]
By definition the space
\begin{align*}
\Tri_{\GL_n}(\tau,\theta^{\GL_n}_{\chi,\psi},\theta^{\GL_n}_{\chi',\psi'})
\end{align*}
of $\GL_n$-equivariant trilinear forms on $\tau\times\theta^{\GL_n}_{\chi,\psi}\times\theta^{\GL_n}_{\chi',\psi'}$ is nonzero. Therefore
\begin{align}\label{space:heredity proof lower space}
\Tri_{M_n}(\tau\otimes\eta,\theta^{\GL_n}_{\chi,\psi}\otimes\theta^{G_0}_{\chi,1},\theta^{\GL_n}_{\chi',\psi'}\otimes\theta^{G_0}_{\chi',1})\ne0.
\end{align}

According to \eqref{eq:containment of Jacquet module along maximal unipotent 2},
\begin{align*}
(\theta_{\chi,\psi}^{G_n})_{U_n}=\theta^{\GL_n}_{|\cdot|^{(n-1)/4}\chi,\psi}\otimes\theta^{G_0}_{\chi,1}.
\end{align*}
Applying Frobenius reciprocity we see that $\theta_{\chi,\psi}^{G_n}$ is a subrepresentation of
\begin{align}\label{eq:tensor in heredity proof 0}
\Ind_{\widetilde{Q}_{n}}^{\widetilde{G}_n}(\delta_{Q_{n}}^{\frac{n-1}{4n}}\theta^{\GL_n}_{\chi,\psi}\otimes\theta^{G_0}_{\chi,1}).
\end{align}
A similar result holds for $\theta^{G_n}_{\chi',\psi'}$.

We define
\begin{align*}
T\in\Tri_{G_n}(\mathrm{I}(\tau,1/2,\eta),\theta^{G_n}_{\chi,\psi},\theta^{G_n}_{\chi',\psi'})
\end{align*}
and prove it is nonzero.
Let $\varphi$ belong to the space $\theta^{G_n}_{\chi,\psi}$, regarded as an element of \eqref{eq:tensor in heredity proof 0}, and similarly let $\varphi'$ belong to the space of $\theta^{G_n}_{\chi',\psi'}$. Also take $f$ in the space of $\mathrm{I}(\tau,1/2,\eta)$.
Now if $L\ne0$ belongs to \eqref{space:heredity proof lower space},
\begin{align*}
L(f(q),\varphi(q),\varphi'(q))=\delta_{Q_n}(q)L(f(1),\varphi(1),\varphi'(1)),\qquad q\in Q_n.
\end{align*}
Thus the following integral is (formally) well defined (see e.g. \cite{BZ1} 1.21),
\begin{align*}
T(f,\varphi,\varphi')=\int_{\lmodulo{Q_n}{G_n}}L(f(g),\varphi(g),\varphi'(g))\ dg.
\end{align*}
It is absolutely convergent according to the Iwasawa decomposition.
Since $T$ satisfies the necessary equivariance properties, it remains to show $T\ne0$.
Assume $L(x,y,y')\ne0$ for suitable data. Take $f$ supported on $Q_{n}\mathcal{N}$,
for a small compact open neighborhood $\mathcal{N}$ of the identity in $G_n$, and such that
\begin{align}
f((a,b)uv)=\delta_{Q_{n}}^{1/2}(a)|\det{a}|^{1/2}\eta(b)\tau(a)x,\qquad\forall (a,b)\in \GL_n\times G_0, u\in U_{n}, v\in\mathcal{N}.
\end{align}
We may assume $\varphi(1)=y$ (because $\theta^{\GL_n}_{\chi,\psi}\otimes\theta^{G_0}_{\chi,1}$ is irreducible) and $\varphi'(1)=y'$.
Using the Iwasawa decomposition and $Q_{n}\mathcal{N}\cap K=(Q_{n}\cap K)\mathcal{N}$ then yields
\begin{align*}
T(f,\varphi,\varphi')=
\int_{(Q_{n}\cap K)\mathcal{N}}L(f(k),\varphi(k),\varphi'(k))\ dk.
\end{align*}
Since $L$ is invariant with respect to $Q_{n}\cap K$, taking a sufficiently small $\mathcal{N}$
(with respect to $\varphi$ and $\varphi'$), the $dk$-integration reduces to a nonzero constant multiple of
$L(f(1),\varphi(1),\varphi'(1))$, which is nonzero.
We conclude that $\mathrm{I}(\tau,1/2,\eta)$ is $(\theta^{G_n}_{\chi,\psi},\theta^{G_n}_{\chi',\psi'})$-distinguished.
\end{proof}

Let $\tau$ be a representation of $G$ as above.
Write $\theta=\theta_{\chi,\psi}$ and $\theta'=\theta_{\chi',\psi'}$. Since $\theta_{\chi,\psi}=\chi\theta_{1,\psi}$,
we may assume $\chi=\chi'=1$, perhaps twisting $\tau$ by a character. For simplicity, we then say that $\tau$ is $(\psi,\psi')$-distinguished.

If $n$ is even, the characters $\psi$ and $\psi'$ can be ignored, because
$\theta_{1,\psi}$ does not depend on $\psi$. If $n$ is odd and $\tau$ is $(\psi_0,\psi_0')$-distinguished,
then for any $\psi$ there is $\psi'$ such that $\tau$ is $(\psi,\psi')$-distinguished. Indeed,
write $\psi(x)=\psi_0(\alpha x)$ for some $\alpha\in F^*$ and put $\psi'(x)=\psi_0'(\alpha x)$, then by
Claim~\ref{claim:taking characters out of exceptional} and its proof,
$\theta_{1,\psi_0}\otimes\theta_{1,\psi_0'}=\theta_{1,\psi}\otimes\theta_{1,\psi'}$.

In light of these observations, we say that
$\tau$ is distinguished if for any $\psi$ there is $\psi'$ such that $\tau$ is $(\psi,\psi')$-distinguished. Proposition~\ref{proposition:upper heredity of dist 1} implies,
\begin{corollary}\label{corollary:upper hered dist}
Let $\tau$ be a distinguished representation of $\GL_n$. Then $\mathrm{I}(\tau,1/2,1)$ is distinguished.
\end{corollary}

Now we prove Theorem~\ref{theorem:tensor of small is non generic}. Namely, for any pair $\theta$ and $\theta'$ of exceptional representations of $\widetilde{G}_n$,
and a generic character $\psi$ of $N_n$,
\begin{align*}
(\theta\otimes\theta')_{N_n,\psi}=0.
\end{align*}
We consider the filtrations of $\theta$ and $\theta'$ corresponding to the Jacquet functor along $C=C_{U_n}$.
The kernel of this functor is glued from representations induced from the Jacquet modules described in Section~\ref{section:Heisenberg jacquet modules}.
Taking the twisted Jacquet functor along $N_n$ truncates some of these quotients and, essentially, reduces the problem to a representation induced from
$(\theta^{G_{2k}})_{C_{U_{2k}},\psi_{k}}\otimes({\theta'}^{G_{2k}})_{C_{U_{2k}},\psi_{k}^{-1}}$.
Theorem~\ref{theorem:small rep is weakly minimal} then enables us to further reduce the problem, to the vanishing of
$\ind_{\Sp_k U_n}^{\GL_{2k}}(\omega_{\psi}\otimes\omega_{\psi^{-1}})_{N_n,\psi}$, which essentially
follows from the results of Offen and Sayag on Klyachko models (\cite{OS}, see also \cite{Klyachko}).
\begin{proof}[Proof of Theorem~\ref{theorem:tensor of small is non generic}]
By an analog of the Geometric Lemma of Bernstein and Zelevinsky (\cite{BZ2} Theorem~5.2 and \cite{BZ1} 5.9-5.12), as a $\widetilde{Q}_n$-module, $\theta$ is glued from
\begin{align*}
\ind_{\widetilde{\mathrm{St}}_{\psi_k}U_n}^{\widetilde{Q}_n}(\theta_{C,\psi_k}), \qquad 0\leq k\leq \lfloor n/2 \rfloor.
\end{align*}
(See Section~\ref{section:Heisenberg jacquet modules} for the notation.)
Then
as a $Q_n$-module $\theta\otimes\theta'$ is glued from
\begin{align}\label{eq:filtration quotients}
\ind_{\widetilde{\mathrm{St}}_{\psi_k}U_n}^{\widetilde{Q}_n}(\theta_{C,\psi_k})\otimes
\ind_{\widetilde{\mathrm{St}}_{\psi_{k'}}U_n}^{\widetilde{Q}_n}(\theta'_{C,\psi_{k'}^{-1}}),\qquad 0\leq
k,k'\leq\lfloor n/2\rfloor.
\end{align}
We prove
that the Jacquet functor with respect to $N_n$ and $\psi$ vanishes on each of these representations.

Since functions in $\ind_{\widetilde{\mathrm{St}}_{\psi_k}U_n}^{\widetilde{Q}_n}(\theta_{C,\psi_k})$ are compactly supported modulo $\widetilde{\mathrm{St}}_{\psi_k}U_n$, and
$C$ is normal in $Q_n$, by Lemma~\ref{lemma:Jacquet kernel as integral},
\begin{align}\label{eq:123 k and k'}
(\ind_{\widetilde{\mathrm{St}}_{\psi_k}U_n}^{\widetilde{Q}_n}(\theta_{C,\psi_k})\otimes
\ind_{\widetilde{\mathrm{St}}_{\psi_{k'}}U_n}^{\widetilde{Q}_n}(\theta'_{C,\psi_{k'}^{-1}}))_{N_n,\psi}=\begin{dcases}
(\ind_{\widetilde{\mathrm{St}}_{\psi_k}U_n}^{\widetilde{Q}_n}(\theta_{C,\psi_k}\otimes\theta'_{C,\psi_{k}^{-1}}))_{N_n,\psi}&k=k',\\0&k\ne k'.
\end{dcases}
\end{align}
To see this consider $f$ in the space of $\ind_{\widetilde{\mathrm{St}}_{\psi_k}U_n}^{\widetilde{Q}_n}(\theta_{C,\psi_k})$ and $f'$ in the space of
$\ind_{\widetilde{\mathrm{St}}_{\psi_{k'}}U_n}^{\widetilde{Q}_n}(\theta'_{C,\psi_{k'}^{-1}})$, and look at the Jacquet-Langlands integral
\begin{align*}
&\int_{\mathcal{C}}c\cdot(f\otimes f')(g,g')\ dc=\int_{\mathcal{C}}f(gc)f'(g'c)\ dc=f(g)f'(g')\int_{\mathcal{C}}\psi_{k}(\rconj{g}c)\psi_{k'}^{-1}(\rconj{g'}c)\ dc,
\end{align*}
where $\mathcal{C}<C$ is a compact subgroup.

Since $\theta_{C,\psi_k}\otimes\theta'_{C,\psi_{k}^{-1}}$ is a non-genuine representation of $\widetilde{\mathrm{St}}_{\psi_k}$, we can replace the representation on the
\rhs\ of \eqref{eq:123 k and k'} with
\begin{align*}
(\ind_{\mathrm{St}_{\psi_k}U_n}^{Q_n}(\theta_{C,\psi_k}\otimes\theta'_{C,\psi_{k}^{-1}}))_{N_n,\psi}.
\end{align*}

Define $\vartheta$ with respect to $\theta^{\GL_{n-2k}}$ as in Proposition~\ref{proposition:Jacquet module C and k} and similarly, define $\vartheta'$ with respect to ${\theta'}^{\GL_{n-2k}}$.
According to the proposition
, $\theta_{C,\psi_k}$ is embedded in a finite direct sum of representations
$\vartheta\otimes(\theta^{G_{2k}})_{C_{U_{2k}},\psi_k}$, which are trivial on $U_{n-2k}$. Put
\begin{align*}
\Pi_k=\vartheta\otimes\vartheta'\otimes(\theta^{G_{2k}})_{C_{U_{2k}},\psi_k}\otimes({\theta'}^{G_{2k}})_{C_{U_{2k}},\psi_k^{-1}}.
\end{align*}
It is enough to prove that for all $0\leq k\leq \lfloor n/2\rfloor$,
\begin{align}\label{eq:vanishing result to prove}
(\ind_{\mathrm{St}_{\psi_k}U_n}^{Q_n}\Pi_k)_{N_n,\psi}=0.
\end{align}
This holds for $k=0$, simply because $U_n$ is normal in $Q_n$, $\Pi_0$ is trivial on $U_n$ while $\psi$ is not.

The case of $k=n/2$ is handled by the following claim, whose proof is deferred to below.
\begin{claim}\label{claim:induced from double Weil claim 0}
Equality~\eqref{eq:vanishing result to prove} holds for $k=n/2$.
\end{claim}

Lastly, assume $0<k<n/2$. Set $Q=Q_{n-2k}\cap Q_n$. By transitivity of induction
\begin{align*}
(\ind_{\mathrm{St}_{\psi_k}U_n}^{Q_n}\Pi_k)_{N_n,\psi}=(\ind_{Q}^{Q_n}(\ind_{\mathrm{St}_{\psi_k}U_n}^{Q}\Pi_k))_{N_n,\psi}.
\end{align*}
The representation $\ind_{\mathrm{St}_{\psi_k}U_n}^{Q}\Pi_k$ is trivial on $U_{n-2k}$.
By virtue of the Geometric Lemma of Bernstein and Zelevinsky (\cite{BZ2} Theorem~5.2), the representation on the \rhs\ is glued from
Jacquet modules of $\ind_{\mathrm{St}_{\psi_k}U_n}^{Q}\Pi_k$. Note that in general, the quotients are representations induced from Jacquet modules, here the induction is trivial because the stabilizer of the character $\psi$ of $N_n$ is $N_n\times G_0$.

Let $\mathcal{W}$ be a set of representatives to the double cosets $\rmodulo{\lmodulo{Q}{Q_n}}{(N_nG_0)}$. That is, $Q_n=\coprod_{w\in\mathcal{W}}Qw^{-1}N_nG_0$. We can take the elements $w$ to be Weyl elements of $\GL_n$. When
$\psi|_{\rconj{w}U_{n-2k}\cap N_n}\ne1$, the quotient corresponding to $w$ vanishes. This implies there is only one quotient, corresponding to $w_0=\left(\begin{smallmatrix}&I_{2k}\\I_{n-2k}\end{smallmatrix}\right)$, namely
\begin{align*}
\delta\cdot(\ind_{\mathrm{St}_{\psi_k}U_n}^{Q}\Pi_k)_{N_{\GL_{n-2k}}\times N_{2k},\psi}.
\end{align*}
Here $\delta$ is some modulus character, hereby ignored, and $N_{\GL_{n-2k}}\times N_{2k}<M_{n-2k}$. As a $G_0$-module, this representation is isomorphic to
\begin{align*}
(\vartheta\otimes\vartheta')_{N_{\GL_{n-2k}},\psi}\otimes\ind_{\mathrm{St}_{\psi_k}U_{2k}}^{Q_{2k}}((\theta^{G_{2k}})_{C_{U_{2k}},\psi_k}\otimes({\theta'}^{G_{2k}})_{C_{U_{2k}},\psi_k^{-1}})_{N_{2k},\psi}.
\end{align*}
Here $\psi$ is regarded as a generic character of $N_{\GL_{n-2k}}$ and $N_{2k}$. Since the case $k=n/2$ implies
\begin{align*}
\ind_{\mathrm{St}_{\psi_k}U_{2k}}^{Q_{2k}}((\theta^{G_{2k}})_{C_{U_{2k}},\psi_k}\otimes({\theta'}^{G_{2k}})_{C_{U_{2k}},\psi_k^{-1}})_{N_{2k},\psi}=0,
\end{align*}
Equality~\eqref{eq:vanishing result to prove} follows.
\begin{proof}[Proof of Claim~\ref{claim:induced from double Weil claim 0}]
For $k=n/2$, $\Pi_k=\theta_{C,\psi_k}\otimes\theta'_{C,\psi_k^{-1}}$.
Now we apply Theorem~\ref{theorem:small rep is weakly minimal}. For simplicity of computations, we can replace
$\psi_k$ with the character defined in the proof of the theorem, then use the epimorphism
$\ell:\widetilde{\mathrm{St}}_{\psi_k}\ltimes U_n\rightarrow \widetilde{\Sp}_k\ltimes H_n$ given there. By
Theorem~\ref{theorem:small rep is weakly minimal}, $\theta_{C,\psi_k}$ is isomorphic to the direct sum of copies
of $\omega_{\psi}$. Pull $\omega_{\psi}$ back to a representation of $\widetilde{\mathrm{St}}_{\psi_k}\ltimes U_n$. Note that
the $G_0$ part of $\mathrm{St}_{\psi_k}$ was ignored in the proof of Theorem~\ref{theorem:small rep is weakly minimal}, since it acts
by a character, so we can ignore this here as well. Equality~\eqref{eq:vanishing result to prove} will follow from
\begin{align}\label{claim:induced from double Weil claim 0 main equality}
(\ind_{\mathrm{St}_{\psi_k}U_{n}}^{Q_{n}}(\omega_{\psi}\otimes\omega_{\psi^{-1}}))_{N_{n},\psi}=0.
\end{align}

We need some notation.  Put
\begin{align*}
\pi_0=\omega_{\psi}\otimes\omega_{\psi^{-1}},\qquad G=P_{n,1}^{\circ},\qquad  V=Z_{n,1},\qquad P=\Sp_k\ltimes V,\qquad X=\lmodulo{P}{G}.
\end{align*}
Note that $\lmodulo{C}{Q_n}\isomorphic G=\GL_n\ltimes V$ (in fact, $\lmodulo{C}{Q_n}\isomorphic G\times G_0$ but $G_0$ was ignored), this isomorphism restricts to an isomorphism $\lmodulo{C}{U_n}\isomorphic V$. In this manner $\psi$ is also a character of $V$, 
$\psi(v)=\psi(v_{n,n+1})$. 
Since $\pi_0$ is trivial on $C$, we can regard it as a representation of $P$ and if
$\pi=\ind_{P}^{G}(\pi_0)$, $\ind_{\mathrm{St}_{\psi_k}U_{n}}^{Q_{n}}(\pi_0)\isomorphic\pi$
as $G$-modules.

We apply the theory of $l$-sheafs of Bernstein and Zelevinsky (\cite{BZ1} 1.13 and Section~6). In the following, we freely use their 
terminology and definitions. Let $(X,\mathcal{F})$ be the $l$-sheaf
corresponding to $\pi$ (\cite{BZ1} 2.23). The group $N_n$ acts on $X$ by $u\cdot x=xu^{-1}$ and on $\mathcal{F}$ by
$u\cdot\varphi(x)=\psi^{-1}(u)\varphi(u^{-1}\cdot x)$. An $N_n$-invariant $\mathcal{F}$-distribution on $X$ is an element of
$\Hom_{N_n}(\pi,\psi)$. Since $\Hom_{N_n}(\pi,\psi)$ is the algebraic dual of $\pi_{N,\psi}$, to prove
\eqref{claim:induced from double Weil claim 0 main equality} we will show that there are no
nonzero $N_n$-invariant $\mathcal{F}$-distributions on $X$.

The action of $N_n$ on $X$ is constructive (\cite{BZ1} Theorem~A). If $x\in X$, let $P^x=\rconj{x^{-1}}P\cap N_n$ be the stabilizer of $x$ in $N_n$. The orbit
of $x$ is $N_n\cdot x$.
The mapping $u\cdot x\mapsto (P^x)u^{-1}$ induces a homeomorphism $N_n\cdot x\isomorphic\lmodulo{P^x}{N_n}$ (\cite{BZ1} 1.6).
The restriction of $\mathcal{F}$ to the orbit of $x$ (this restriction is an $l$-sheaf, because the action is constructive) is isomorphic to
$\ind_{P^x}^{N_n}(\rconj{x^{-1}}\pi_0)$, where $\rconj{x^{-1}}\pi_0$ is the representation
of $P^x$ acting in the space of $\pi_0$ by $\rconj{x^{-1}}\pi_0(z)=\pi_0(\rconj{x}z)$.

By virtue of Theorem~6.9 of \cite{BZ1}, to show there are no nonzero $N_n$-invariant $\mathcal{F}$-distributions on $X$, it is enough to
prove that for each representative $x$,
\begin{align*}
\Hom_{N_n}(\ind_{P^x}^{N_n}(\rconj{x^{-1}}\pi_0),\psi)=\Hom_{P^x}(\rconj{x^{-1}}\pi_0,\psi)=0.
\end{align*}
We can take representatives $x\in P_{n-1,1}^{\circ}<\GL_n$. Then $P^x=\Sp_k^x\ltimes V$, where $\Sp_k^x=\rconj{x^{-1}}\Sp_k\cap N_{\GL_n}$.
In addition, because $P_{n-1,1}^{\circ}$ fixes $\psi|_{V}$, $(\rconj{x^{-1}}\pi_0)_{V,\psi}=\rconj{x^{-1}}({(\pi_0)}_{V,\psi})$. Hence
\begin{align*}
\Hom_{P^x}(\rconj{x^{-1}}\pi_0,\psi)=\Hom_{\Sp_k^x}(\rconj{x^{-1}}({(\pi_0)}_{V,\psi}),\psi).
\end{align*}

Note that $(\pi_0)_{V,\psi}=(\omega_{\psi}\otimes\omega_{\psi^{-1}})_{H_n,\mu}$ for the nontrivial $\mu$ given in
Claim~\ref{claim:Jacquet modules along Hn of double Weil reps}. According to that claim $(\pi_0)_{V,\psi}$ is the trivial one-dimensional  
representation of $\ell^{-1}(\Sp_k\cap P_{n-1,1}^{\circ})$. Since 
$\rconj{x}N_{\GL_n}<P_{n-1,1}^{\circ}$ for any $x\in P_{n-1,1}^{\circ}$, 
$\rconj{x^{-1}}({(\pi_0)}_{V,\psi})$ is trivial on $\Sp_k^x$ (the epimorphism $\ell$ is easily seen to be harmless here). 
But Offen and Sayag (\cite{OS} Proposition~2,
we use $\mathcal{H}^{r,r'}$ with $r=0$ and $r'=n$, in their notation) proved that
$\psi|_{\Sp_k^x}\ne1$ for any $x\in \GL_n$. This implies $\Hom_{\Sp_k^x}(\rconj{x^{-1}}({(\pi_0)}_{V,\psi}),\psi)=0$, as required.
\end{proof}
\end{proof}

\begin{corollary}\label{corollary:upper hered dist LQ}
Let $\tau$ be an irreducible unitary supercuspidal $(\theta^{\GL_n}_{\chi,\psi},\theta^{\GL_n}_{\chi',\psi'})$-distinguished representation of $\GL_n$.
Then the Langlands quotient of $\mathrm{I}(\tau,1/2,(\chi\chi')^{-1})$ is
$(\theta^{G_n}_{\chi,\psi},\theta^{G_n}_{\chi',\psi'})$-distinguished. In particular, if $\tau$ is
distinguished, so is the Langlands quotient of $\mathrm{I}(\tau,1/2,1)$.
\end{corollary}
\begin{proof}[Proof of Corollary~\ref{corollary:upper hered dist LQ}]
According to Proposition~\ref{proposition:upper heredity of dist 1},
\begin{align*}
\Hom_{G_n}(\theta^{G_n}_{\chi,\psi}\otimes\theta^{G_n}_{\chi',\psi'},\mathrm{I}(\tau,1/2,(\chi\chi')^{-1})^{\vee})\ne0.
\end{align*}
Theorem~\ref{theorem:tensor of small is non generic} implies $\mathrm{I}(\tau,1/2,(\chi\chi')^{-1})$ is reducible (because $\tau$ is generic),
then by Casselman and Shahidi \cite{CSh} (Theorem~1), the Langlands quotient $\mathrm{LQ}(\mathrm{I}(\tau,1/2,(\chi\chi')^{-1}))$ is non-generic, and
the unique irreducible subspace of $\mathrm{I}(\tau,1/2,(\chi\chi')^{-1})$ is generic. Also note that the length of $\mathrm{I}(\tau,1/2,(\chi\chi')^{-1})$
is $2$ (because it is reducible and $\tau$ is supercuspidal, see \cite{BZ2}~2.8). Now the result follows from
Theorem~\ref{theorem:tensor of small is non generic} and the left exactness of the $\Hom$ functor.
\end{proof}

As a corollary, we can now prove Theorem~\ref{theorem:supercuspidal dist GLn and pole}. Namely,
for an irreducible unitary supercuspidal $\tau$, being distinguished is equivalent to the occurrence of a pole at $s=0$ of
$L(s,\tau,\mathrm{Sym}^2)$.
\begin{proof}[Proof of Theorem~\ref{theorem:supercuspidal dist GLn and pole}]
If $\tau$ is distinguished, as in the proof of Corollary~\ref{corollary:upper hered dist LQ} we see that $\mathrm{I}(\tau,1/2,1)$ is reducible,
then according to Casselman and Shahidi \cite{CSh} (Proposition~5.3, their Conjecture 1.1 was proved for $G_n$ in \cite{Asg}) $L(s,\tau,\mathrm{Sym}^2)$ has a pole at $s=0$.

In the other direction, assume $L(s,\tau,\mathrm{Sym}^2)$ has a pole at $s=0$. As an application of the descent method of Ginzburg, Rallis and Soudry 
(see e.g., \cite{GRS2,GRS5,GRS7,GRS3,GRS8,JSd1,JSd2,Soudry5,Soudry4,RGS}), one can globalize $\tau$ to a cuspidal automorphic representation
$\pi$ of $\GL_n(\Adele)$, such that $L^S(s,\pi,\mathrm{Sym}^2)$ has a pole at $s=1$ (see the appendix of \cite{PR}).
Therefore the nonvanishing of \eqref{int:BG period} implies $\tau$ is distinguished (\cite{BG} Theorem~7.6).
\end{proof}
\begin{corollary}
If $\tau$ is an irreducible unitary supercuspidal distinguished representation, it must be self-dual.
\end{corollary}
\begin{remark}\label{remark:applicability to twisted}
The analogous result for an irreducible supercuspidal $(\theta^{\GL_n}_{\chi,\psi},\theta^{\GL_n}_{\chi',\psi'})$-distinguished $\tau$ should
also hold. One needs to verify the applicability of the globalization argument (see \cite{HS}).
\end{remark}

Given exceptional representations $\theta$ and $\theta'$ of $\widetilde{G}$ (as in the beginning of this section), we can consider the space of $\theta\otimes\theta'$ as a model for representations of $G$.
We refer to the dimension of $\Hom_G(\theta\otimes\theta',\tau^{\vee})$ as the multiplicity of $\tau$. The next proposition relates the multiplicities
of $\tau$ and $\mathrm{I}(\tau,1/2,\eta)$.

In the case of $\GL_n$, Kable \cite{Kable} conjectured that the multiplicity of an irreducible representation is at most one. He proved this for $n\leq 3$, and for arbitrary $n$
under a certain homogeneity condition (\cite{Kable} Corollary~6.1). 
It is reasonable to believe multiplicity one also holds in the context of $G_n$ (see
\cite{LM6} Remark~4.2).
\begin{proposition}\label{proposition:one-dim}
Let $\tau$ be an irreducible tempered representation of $\GL_n$, put $\eta=(\chi\chi')^{-1}$ and assume $|\eta|=1$. Then
\begin{align}\label{eq:proposition one-dim statement}
\Dim\ \Hom_{G_n}(\theta^{G_n}_{\chi,\psi}\otimes\theta^{G_n}_{\chi',\psi'},\mathrm{I}(\tau,1/2,\eta)^{\vee})=\Dim\
\Hom_{\GL_n}(\theta^{\GL_n}_{\chi,\psi}\otimes\theta^{\GL_n}_{\chi',\psi'},\tau^{\vee}).
\end{align}
In particular, the representation $\tau$ is $(\theta^{\GL_n}_{\chi,\psi},\theta^{\GL_n}_{\chi',\psi'})$-distinguished if and only if
$\mathrm{I}(\tau,1/2,\eta)$ is $(\theta^{G_n}_{\chi,\psi},\theta^{G_n}_{\chi',\psi'})$-distinguished.
\end{proposition}
\begin{remark}\label{remark:uniqueness current state}
If $\tau$ is $(\theta^{\GL_n}_{\chi,\psi},\theta^{\GL_n}_{\chi',\psi'})$-distinguished, 
in particular $\omega_{\tau}(z^2\cdot I_n)=\eta(z^{2n})$ for all $z\in F^*$ ($\omega_{\tau}$ - the central character of $\tau$). It follows that $|\eta|=1$, because a tempered representation is 
unitary. In this case by 
\eqref{eq:proposition one-dim statement} 
$\tau$ enjoys multiplicity one if and only if $\mathrm{I}(\tau,1/2,\eta)$ does.
\end{remark}
\begin{proof}[Proof of Proposition~\ref{proposition:one-dim}]
Set $\theta_0=\theta^{\GL_n}_{\chi,\psi}$, $\theta=\theta^{G_n}_{\chi,\psi}$ and similarly $\theta_0'$ and $\theta'$ (with $\chi'$ and $\psi'$).
For brevity, put $\mathrm{d}=|\det{}|$. Applying Frobenius reciprocity,
\begin{align*}
\Hom_{G_n}(\theta\otimes\theta',\mathrm{I}(\tau,1/2,\eta)^{\vee})=
\Hom_{\GL_n}((\theta\otimes\theta')_{U_n}|_{\GL_n},\mathrm{d}^{(n-1)/2}\tau^{\vee}).
\end{align*}
Using the notation of Section~\ref{section:Heisenberg jacquet modules}, the representation $(\theta\otimes\theta')_{U_n}$ is filtered by representations
\begin{align*}
\mathbb{W}_k=(\ind_{\mathrm{St}_{\psi_k}U_n}^{Q_n}(\theta_{C,\psi_k}\otimes\theta'_{C,\psi_{k}^{-1}}))_{U_n},\qquad 0\leq k\leq \lfloor n/2 \rfloor.
\end{align*}
The representation
$\mathbb{W}_0$ is a quotient of $(\theta\otimes\theta')_{U_n}$ and the kernel of the mapping
$(\theta\otimes\theta')_{U_n}\rightarrow\mathbb{W}_0$ is filtered by $\mathbb{W}_k$ with $1\leq k\leq \lfloor n/2 \rfloor$. Regard $\mathbb{W}_k$ as a representation of $\GL_n$ (by restriction from $M_n$).

According to
Proposition~\ref{proposition:Jacquet module C and k},
$\mathbb{W}_0=\mathrm{d}^{(n-1)/2}\theta_0\otimes\theta_0'$. Then
\begin{align*}
\Hom_{\GL_n}(\mathbb{W}_0
,\mathrm{d}^{(n-1)/2}\tau^{\vee})=
\Hom_{\GL_n}(\theta_0\otimes\theta_0',\tau^{\vee}).
\end{align*}
Thus \eqref{eq:proposition one-dim statement}
will follow, once we prove that for $k>0$,
\begin{align}\label{rep:homspace quotient in filtration one-dim}
\Hom_{\GL_n}(\mathbb{W}_k,
\mathrm{d}^{(n-1)/2}\tau^{\vee})=0.
\end{align}

If $n$ is even, we claim the following.
\begin{claim}\label{claim:induced from double Weil claim 1}
For any irreducible generic representation $\sigma$ of $\GL_n$, $\Hom_{\GL_n}(\mathbb{W}_{n/2},\sigma)=0$. In particular 
Equality~\eqref{rep:homspace quotient in filtration one-dim} holds for $k=n/2$. 
\end{claim}
The proof will be given below. 

Assume $0<k<n/2$. The representation $\theta_{C,\psi_k}$ is embedded in a finite direct sum of representations
$\vartheta\otimes(\theta^{G_{2k}})_{C_{U_{2k}},\psi_k}$ (Proposition~\ref{proposition:Jacquet module C and k}). The representation $\vartheta$ is semisimple, and
by Theorem~\ref{theorem:small rep is weakly minimal} the representation $(\theta^{G_{2k}})_{C_{U_{2k}},\psi_k}$ is also semisimple.
Hence $\theta_{C,\psi_k}\otimes\theta'_{C,\psi_{k}^{-1}}$ is embedded in a semisimple representation, which is the finite direct sum of
$\Pi_k$ ($\Pi_k=\vartheta\otimes\vartheta'\otimes(\theta^{G_{2k}})_{C_{U_{2k}},\psi_k}\otimes({\theta'}^{G_{2k}})_{C_{U_{2k}},\psi_k^{-1}}$).
Therefore, using the exactness of induction and Jacquet functors, $\mathbb{W}_k$ is a quotient of
$(\ind_{\mathrm{St}_{\psi_k}U_n}^{Q_n}(\bigoplus\Pi_k))_{U_n}$. We conclude that we may replace $\mathbb{W}_k$ in \eqref{rep:homspace quotient in filtration one-dim} with $(\ind_{\mathrm{St}_{\psi_k}U_n}^{Q_n}\Pi_k)_{U_n}$.
An application of \cite{BZ2} (Theorem~5.2) yields
\begin{align*}
(\ind_{\mathrm{St}_{\psi_k}U_n}^{Q_n}\Pi_k)_{U_n}=
\ind_{P_{n-2k,2k}}^{\GL_n}(\vartheta\otimes\vartheta'\otimes\mathbb{W}'),
\end{align*}
where
\begin{align*}
\mathbb{W}'=
(\ind_{\mathrm{St}_{\psi_k}U_{2k}}^{Q_{2k}}(\theta^{G_{2k}}_{C_{U_{2k}},\psi_k}\otimes{\theta'}^{G_{2k}}_{C_{U_{2k}},\psi_{k}^{-1}}))_{U_{2k}}|_{\GL_{2k}}.
\end{align*}
Note that the relevant double coset space (for \cite{BZ2} Theorem~5.2) is $\rmodulo{\lmodulo{(Q_{n-2k}\cap Q_n)}{Q_n}}{Q_n}$, there is only one representative to consider.

Hence to deduce \eqref{rep:homspace quotient in filtration one-dim}, it is enough to prove
\begin{align*}
\Hom_{\GL_n}(\ind_{P_{n-2k,2k}}^{\GL_n}(\vartheta\otimes\vartheta'\otimes
\mathbb{W}'),\mathrm{d}^{(n-1)/2}\tau^{\vee})=0.
\end{align*}
The \lhs\ equals
\begin{align}\nonumber
&\Hom_{\GL_n}(\mathrm{d}^{(1-n)/2}\tau,\ind_{P_{n-2k,2k}}^{\GL_n}(\delta_{P_{n-2k,2k}}(\vartheta\otimes\vartheta'\otimes\mathbb{W}')^{\vee}))\nonumber
\\&=\Hom_{P_{n-2k,2k}}(\mathrm{d}^{(1-n)/2}\tau_{Z_{n-2k,2k}},\delta_{P_{n-2k,2k}}(\vartheta\otimes\vartheta'\otimes\mathbb{W}')^{\vee})\nonumber
\\\label{eq:last space}&=\Hom_{P_{n-2k,2k}}(\mathrm{d}^{(1-n)/2}\delta_{P_{n-2k,2k}}^{-1/2}\vartheta\otimes\vartheta'\otimes\mathbb{W}',j(\tau)^{\vee}).
\end{align}
Here $j(\tau)=\delta_{P_{n-2k,2k}}^{-1/2}\tau_{Z_{n-2k,2k}}$ (the normalized Jacquet functor). It suffices to show that
the last space \eqref{eq:last space} vanishes, when $j(\tau)$ is replaced by any of its irreducible subquotients.
Let $\tau_1\otimes\tau_2$ be one such subquotient. Then \eqref{eq:last space} vanishes if either
\begin{align}
\label{eq:last tensor 1}&\Hom_{\GL_{2k}}(\mathbb{W}',\mathrm{d}^{(2k-1)/2}\tau_2^{\vee})\text{ or}\\
\label{eq:last tensor 2}&\Hom_{\GL_{n-2k}}(\mathrm{d}^{(1-n-2k)/2}\vartheta\otimes\vartheta',\tau_1^{\vee})
\end{align}
are zero. As explained in the proof of Proposition~4.1 of Lapid and Mao \cite{LM6}, since $\tau$ is tempered, either $\tau_2$ is generic,
or the central character $\omega_{\tau_1}$ of $\tau_1$ satisfies $|\omega_{\tau_1}|=|\det{}|^{\alpha}$ for some $\alpha>0$. In the former case
\eqref{eq:last tensor 1} vanishes by Claim~\ref{claim:induced from double Weil claim 1}. In the latter case \eqref{eq:last tensor 2} vanishes. Indeed,
$z^2\cdot I_{n-2k}$ acts on $\vartheta\otimes\vartheta'$ by $\mathrm{d}^{(-1+n+2k)/2}\eta^{-1}(z^{2(n-2k)})$ and since
$|\eta|=1$, this action is unitary on
$\mathrm{d}^{(1-n-2k)/2}\vartheta\otimes\vartheta'$, but $\omega_{\tau_1}$ is not unitary.
\begin{proof}[Proof of Claim~\ref{claim:induced from double Weil claim 1}]
Put $k=n/2$ and $\pi=\ind_{\mathrm{St}_{\psi_k}U_n}^{Q_n}(\theta_{C,\psi_k}\otimes\theta'_{C,\psi_k^{-1}})$. We need to prove 
\begin{align*}
\Hom_{\GL_n}(\pi_{U_n},
\sigma)=0.
\end{align*}
We will show that $\pi_{U_n}$ has a filtration,
whose quotients are all isomorphic to $\ind_{\Sp_{n/2}}^{\GL_n}1$. 
Then since $\sigma$ is irreducible and generic, $\Hom_{\GL_n}(\ind_{\Sp_{n/2}}^{\GL_n}1,\sigma)=0$ 
(\cite{OS} Proposition~1, take $\mathcal{H}^{0,n}$) and the result follows. 

We turn to prove the filtration of $\pi_{U_n}$. As in the proof of Claim~\ref{claim:induced from double Weil claim 0}, 
Theorem~\ref{theorem:small rep is weakly minimal} implies that $\pi_{U_n}$ is filtered by copies of the representation 
$(\ind_{\mathrm{St}_{\psi_k}U_{n}}^{Q_{n}}(\omega_{\psi}\otimes\omega_{\psi^{-1}}))_{U_n}$. We prove
\begin{align}\label{eq:claim desp 1}
(\ind_{\mathrm{St}_{\psi_k}U_{n}}^{Q_{n}}(\omega_{\psi}\otimes\omega_{\psi^{-1}}))_{U_n}=\ind_{\Sp_{n/2}}^{\GL_n}1.
\end{align}
(Recall that the \lhs\ is regarded as a representation of $\GL_n$.) 

Since $U_n$ is normal in $Q_n$, Lemma~\ref{lemma:Jacquet kernel as integral} implies
\begin{align*}
(\ind_{\mathrm{St}_{\psi_k}U_{n}}^{Q_{n}}(\omega_{\psi}\otimes\omega_{\psi^{-1}}))_{U_n}=
\ind_{\mathrm{St}_{\psi_k}U_{n}}^{Q_{n}}((\omega_{\psi}\otimes\omega_{\psi^{-1}})_{H_n}).
\end{align*}
In more detail, if $f$ belongs to the space of $\ind_{\mathrm{St}_{\psi_k}U_{n}}^{Q_{n}}(\omega_{\psi}\otimes\omega_{\psi^{-1}})$,
the Jacquet-Langlands integral takes the form
\begin{align*}
\int_{\mathcal{U}}f(gu)\ du=\int_{\mathcal{U}}
(\omega_{\psi}\otimes\omega_{\psi^{-1}})(\rconj{g}u)f(g)\ du,
\end{align*}
for a compact subgroup $\mathcal{U}<U_n$. Because $f$ is compactly supported modulo $\mathrm{St}_{\psi_k}U_{n}$, this integral vanishes for all $g\in Q_n$ if and only if
$f(g)$ belongs to the space of $(\omega_{\psi}\otimes\omega_{\psi^{-1}})({H_n})$ for all $g$. It remains to use the exactness of induction.

According to Claim~\ref{claim:Jacquet modules along Hn of double Weil reps}, $(\omega_{\psi}\otimes\omega_{\psi^{-1}})_{H_n}$ is the trivial one-dimensional 
representation of $\Sp_k$ and 
\eqref{eq:claim desp 1} holds.
\end{proof}
\end{proof}
We can now improve Corollary~\ref{corollary:upper hered dist} for tempered representations:
\begin{corollary}\label{corollary:upper hered dist for tempered}
Let $\tau$ be an irreducible tempered representation of $\GL_n$. Then
$\tau$ is distinguished if and only if $\mathrm{I}(\tau,1/2,1)$ is distinguished.
\end{corollary}

\section{The small representation of $\SO_{2n+1}$}\label{section:relevance to SO(2n+1)}
Bump, Friedberg and Ginzburg \cite{BFG} constructed and studied the
small representations for $\SO_{2n+1}$. In this section we briefly recall their results and formulate our results for $\SO_{2n+1}$.
The cover $\widetilde{SO}_{2n+1}$ was obtained by restricting the $4$-fold cover of
$\SL_{2n+1}$ of Matsumoto \cite{Mats}. It is a ``double cover"
in the sense that the square of the cocycle is trivial on the kernel of the spinor norm.

We use the same notation of $G_n$, e.g., $B_n=B_{\SO_{2n+1}}$ (see Section~\ref{subsection:GSpin}), $T_n$ is the diagonal torus and $Q_k$ is a standard maximal parabolic subgroup.
If $(a,g)\in\GL_k\times SO_{2(n-k)+1}$, $(a,g)$ is embedded in $M_k$ as $\mathrm{diag}(a,g,J_k\transpose{a}^{-1}J_k)$.
The block compatibility formula now reads (see \cite{BFG} (2.20))
\begin{align*}
&\sigma_{\SL_{2n+1}}((a,g),(a',g'))\\&=\sigma_{\SL_{k+1}}(\mathrm{diag}(a,\det{a}^{-1}),\mathrm{diag}(a',\det{a'}^{-1}))^2(\det{a},\det{a'})_4\sigma_{\SL_{2(n-k)+1}}(g,g').
\end{align*}
The benefit of this cover is that the preimages
of $\GL_k$ and $SO_{2(n-k)+1}$ commute, thus
the tensor can be used to describe representations of Levi subgroups. Restriction of the cover to $\GL_k$ is a double cover.

The small representation $\theta=\theta^{\SO_{2n+1}}$ is unique; it is the representation of $\widetilde{SO}_{2n+1}$ corresponding to the
exceptional character
\begin{align*}
\xi(\mathfrak{s}(\mathrm{diag}(t_1^2,\ldots,t_n^2,1,t_n^{-2},\ldots,t_1^{-2})))=\prod_{i=1}^n|t_i|^{n-i+1}.
\end{align*}
According to \cite{BFG} (Theorem~2.3), $\theta_{U_k}=\theta^{\GL_k}\otimes\theta^{SO_{2(n-k)+1}}$, where $\theta^{\GL_k}$ was explicitly
given, and by \cite{BFG} (Theorem~2.6) and \cite{BFG2} (Proposition~3), $\theta_{U_1,\psi^{(1)}}=0$ if the length of $\psi^{(1)}$ is nonzero
(using the notation of Section~\ref{subsection:The exceptional representations}).

Theorems~\ref{theorem:tensor of small is non generic} and \ref{theorem:small rep is weakly minimal} remain valid as stated. Proposition~\ref{proposition:Jacquet module C and k} now takes the form
\begin{align*}
\theta_{C,\psi_k}=\theta^{\GL_{n-2k}}\otimes(\theta^{SO_{4k+1}})_{C_{U_{2k}},\psi_k}.
\end{align*}
Here $\theta^{\GL_{n-2k}}$ is uniquely determined.
Indeed, this equality replaces \eqref{eq:equality to correct if k = 0} because $\theta_{U_{n-2k}}$ is a tensor of representations.

Regarding Proposition~\ref{proposition:upper heredity of dist 1}, assume $\tau$ is $(\psi,\psi')$-distinguished. Twisting $\tau$ by some square trivial character,
we obtain a $(\psi,\psi)$-distinguished representation. Then, perhaps after using another twist of $\tau$, it becomes $(\theta^{\GL_n},\theta^{\GL_n})$-distinguished
for the exceptional representation $\theta^{\GL_n}$ such that $\theta_{U_n}=\theta^{\GL_n}$. Now the proof of the proposition proceeds as in the case
of $G_n$. So, in this case one must start with a $(\theta^{\GL_n},\theta^{\GL_n})$-distinguished representation of $\GL_n$, in order to obtain
a distinguished representation $\mathrm{I}(\tau,1/2,1)$ of $\SO_{2n+1}$. Corollary~\ref{corollary:upper hered dist LQ} is applicable
for $\tau$ such that Proposition~\ref{proposition:upper heredity of dist 1} is valid. Proposition~\ref{proposition:one-dim} 
remains valid.


\begin{thebibliography}{MVW87}

\bibitem[AS06]{AsgSha}
M.~Asgari and F.~Shahidi.
\newblock Generic transfer for general spin groups.
\newblock {\em Duke Math. J.}, 132(1):137--190, 2006.

\bibitem[AS11]{AsgSha2}
M.~Asgari and F.~Shahidi.
\newblock Functoriality for general spin groups, 2011.
\newblock arXiv:1101.3467 [math.NT], {http://arxiv.org/abs/1101.3467}.

\bibitem[Asg00]{Asg2}
M.~Asgari.
\newblock {\em On holomorphy of local {L}anglands {L}-functions}.
\newblock ProQuest LLC, Ann Arbor, MI, 2000.
\newblock Thesis (Ph.D.)--Purdue University.

\bibitem[Asg02]{Asg}
M.~Asgari.
\newblock Local {$L$}-functions for split spinor groups.
\newblock {\em Canad. J. Math.}, 54(4):673--693, 2002.

\bibitem[BFG00]{BFG3}
D.~Bump, S.~Friedberg, and D.~Ginzburg.
\newblock A {R}ankin-{S}elberg integral using the automorphic minimal
  representation of {${\rm SO}(7)$}.
\newblock {\em J. Ramanujan Math. Soc.}, 15(2):81--124, 2000.

\bibitem[BFG03]{BFG}
D.~Bump, S.~Friedberg, and D.~Ginzburg.
\newblock Small representations for odd orthogonal groups.
\newblock {\em Internat. Math. Res. Notices}, (25):1363--1393, 2003.

\bibitem[BFG06]{BFG2}
D.~Bump, S.~Friedberg, and D.~Ginzburg.
\newblock Lifting automorphic representations on the double covers of
  orthogonal groups.
\newblock {\em Duke Math. J.}, 131(2):363--396, 2006.

\bibitem[BG92]{BG}
D.~Bump and D.~Ginzburg.
\newblock Symmetric square {$L$}-functions on {${\rm GL}(r)$}.
\newblock {\em Ann. of Math. (2)}, 136(1):137--205, 1992.

\bibitem[BK94]{BK}
R.~Brylinski and B.~Kostant.
\newblock Minimal representations of {$E_6$}, {$E_7$}, and {$E_8$} and the
  generalized {C}apelli identity.
\newblock {\em Proc. Nat. Acad. Sci. U.S.A.}, 91(7):2469--2472, 1994.

\bibitem[BLS99]{BLS}
W.~D. Banks, J.~Levy, and M.~R. Sepanski.
\newblock Block-compatible metaplectic cocycles.
\newblock {\em J. Reine Angew. Math.}, 507:131--163, 1999.

\bibitem[BS98]{BSch}
R.~Berndt and R.~Schmidt.
\newblock {\em Elements of the representation theory of the {J}acobi group},
  volume 163 of {\em Progress in Mathematics}.
\newblock Birkh\"auser Verlag, Basel, 1998.

\bibitem[BZ76]{BZ1}
I.~N. Bernstein and A.~V. Zelevinsky.
\newblock Representations of the group ${GL(n,F)}$ where ${F}$ is a local
  non-{A}rchimedean field.
\newblock {\em Russian Math. Surveys}, 31(3):1--68, 1976.

\bibitem[BZ77]{BZ2}
I.~N. Bernstein and A.~V. Zelevinsky.
\newblock Induced representations of reductive ${p}$-adic groups {I}.
\newblock {\em Ann. Scient. \'{E}c. Norm. Sup.}, 10(4):441--472, 1977.

\bibitem[Car93]{Cr}
R.~W. Carter.
\newblock {\em Finite groups of {L}ie type}.
\newblock Wiley Classics Library. John Wiley \& Sons Ltd., Chichester, 1993.
\newblock Conjugacy classes and complex characters, Reprint of the 1985
  original, A Wiley-Interscience Publication.

\bibitem[CM93]{CM}
D.~H. Collingwood and W.~M. McGovern.
\newblock {\em Nilpotent orbits in semisimple {L}ie algebras}.
\newblock Van Nostrand Reinhold Mathematics Series. Van Nostrand Reinhold Co.,
  New York, 1993.

\bibitem[CS98]{CSh}
W.~Casselman and F.~Shahidi.
\newblock On irreducibility of standard modules for generic representations.
\newblock {\em Ann. Scient. \'{E}c. Norm. Sup.}, 31(4):561--589, 1998.

\bibitem[FK86]{FK}
Y.~Z. Flicker and D.~A. Kazhdan.
\newblock Metaplectic correspondence.
\newblock {\em Inst. Hautes \'Etudes Sci. Publ. Math.}, 64:53--110, 1986.

\bibitem[Gin90]{G}
D.~Ginzburg.
\newblock ${L}$-functions for ${SO_n\times GL_k}$.
\newblock {\em J. Reine Angew. Math.}, 405:156--180, 1990.

\bibitem[Gin06]{G2}
D.~Ginzburg.
\newblock Certain conjectures relating unipotent orbits to automorphic
  representations.
\newblock {\em Israel J. Math.}, 151:323--355, 2006.

\bibitem[GJS10]{GJS}
D.~Ginzburg, D.~Jiang, and D.~Soudry.
\newblock Periods of automorphic forms, poles of {$L$}-functions and functorial
  lifting.
\newblock {\em Sci. China Math.}, 53(9):2215--2238, 2010.

\bibitem[GJS11]{GJS2}
D.~Ginzburg, D.~Jiang, and D.~Soudry.
\newblock On certain automorphic descents to {${\rm GL}_2$}.
\newblock {\em Internat. Math. Res. Notices}, 2011(21):4779--4820, 2011.

\bibitem[GRS97a]{GRS}
D.~Ginzburg, S.~Rallis, and D.~Soudry.
\newblock On the automorphic theta representation for simply laced groups.
\newblock {\em Israel J. Math.}, 100:61--116, 1997.

\bibitem[GRS97b]{GRS2}
D.~Ginzburg, S.~Rallis, and D.~Soudry.
\newblock Self-dual automorphic ${GL_n}$ modules and construction of a backward
  lifting from ${GL_n}$.
\newblock {\em Internat. Math. Res. Notices.}, 687(14):687--–701, 1997.

\bibitem[GRS99a]{GRS5}
D.~Ginzburg, S.~Rallis, and D.~Soudry.
\newblock On a correspondence between cuspidal representations of ${GL_{2n}}$
  and $\widetilde{Sp}_{2n}$.
\newblock {\em J. Amer. Math. Soc.}, 12(3):849--907, 1999.

\bibitem[GRS99b]{GRS7}
D.~Ginzburg, S.~Rallis, and D.~Soudry.
\newblock On explicit lifts of cusp forms from {${\rm GL}_m$} to classical
  groups.
\newblock {\em Ann. of Math. (2)}, 150(3):807--866, 1999.

\bibitem[GRS01]{GRS3}
D.~Ginzburg, S.~Rallis, and D.~Soudry.
\newblock Generic automorphic forms on ${SO(2n+1)}$: functorial lift to
  ${GL(2n)}$, endoscopy, and base change.
\newblock {\em Internat. Math. Res. Notices.}, 2001(14):729--764, 2001.

\bibitem[GRS03]{GRS6}
D.~Ginzburg, S.~Rallis, and D.~Soudry.
\newblock On {F}ourier coefficients of automorphic forms of symplectic groups.
\newblock {\em Manuscripta Math.}, 111(1):1--16, 2003.

\bibitem[GRS11]{RGS}
D.~Ginzburg, S.~Rallis, and D.~Soudry.
\newblock {\em The Descent Map from Automorphic Representations of {GL(n)} to
  Classical Groups}.
\newblock World Scientific Publishing, Singapore, 2011.

\bibitem[GS05]{GanSavin}
W.~T. Gan and G.~Savin.
\newblock On minimal representations definitions and properties.
\newblock {\em Represent. Theory}, 9:46--93 (electronic), 2005.

\bibitem[GSS02]{GRS8}
D.~Ginzburg, Rallis S, and D.~Soudry.
\newblock Endoscopic representations of {${\widetilde{\rm Sp}}_{2n}$}.
\newblock {\em J. Inst. Math. Jussieu}, 1(1):77--123, 2002.

\bibitem[HS12]{HS}
J.~Hundley and E.~Sayag.
\newblock Descent construction for {GS}pin groups, 2012.
\newblock arXiv:1110.6788v2 [math.NT], {http://arxiv.org/pdf/1110.6788.pdf}.

\bibitem[Ike94]{Ik3}
T.~Ikeda.
\newblock On the theory of {J}acobi forms and {F}ourier-{J}acobi coefficients
  of {E}isenstein series.
\newblock {\em J. Math. Kyoto Univ.}, 34(3):615--636, 1994.

\bibitem[JR92]{JR}
H.~Jacquet and S.~Rallis.
\newblock Symplectic periods.
\newblock {\em J. Reine Angew. Math.}, 423:175--197, 1992.

\bibitem[JR96]{JR2}
H.~Jacquet and S.~Rallis.
\newblock Uniqueness of linear periods.
\newblock {\em Compositio Math.}, 102(1):65--123, 1996.

\bibitem[JS03]{JSd1}
D.~Jiang and D.~Soudry.
\newblock The local converse theorem for ${SO(2n+1)}$ and applications.
\newblock {\em Ann. of Math.}, 157(3):743--806, 2003.

\bibitem[JS04]{JSd2}
D.~Jiang and D.~Soudry.
\newblock Generic representations and local {L}anglands reciprocity law for
  ${p}$-adic ${SO(2n+1)}$.
\newblock In H.~Hida, D.~Ramakrishnan, and F.~Shahidi, editors, {\em
  Contributions to automorphic forms, geometry, and number theory: a volume in
  honor of {J}oseph {S}halika}, pages 457--519, Baltimore, MD, 2004. Johns
  Hopkins Univ. Press.

\bibitem[Kab01]{Kable}
A.~C. Kable.
\newblock The tensor product of exceptional representations on the general
  linear group.
\newblock {\em Ann. Sci. \'Ecole Norm. Sup. (4)}, 34(5):741--769, 2001.

\bibitem[Kab02]{Kable2}
A.~C. Kable.
\newblock On a conjecture of {S}avin.
\newblock {\em Internat. Math. Res. Notices.}, 2002(30):1601--1627, 2002.

\bibitem[Kap]{me7}
E.~Kaplan.
\newblock The theta period of a cuspidal automorphic representation of
  ${GL(n)}$.
\newblock 2014, doi:10.1093/imrn/rnt358, to appear in Internat. Math. Res.
  Notices.

\bibitem[Kap14a]{me8}
E.~Kaplan.
\newblock The double cover of odd general spin groups, small representations
  and applications.
\newblock arXiv:1411.4082 [math.RT], {http://arxiv.org/pdf/1411.4082.pdf},
  2014.

\bibitem[Kap14b]{me9}
E.~Kaplan.
\newblock Representations distinguished by pairs of exceptional representations
  and a conjecture of {S}avin.
\newblock arXiv:1411.3697 [math.RT], {http://arxiv.org/pdf/1411.3697.pdf},
  2014.

\bibitem[Kaz90]{KZ}
D.~Kazhdan.
\newblock The minimal representation of {$D_4$}.
\newblock In {\em Operator algebras, unitary representations, enveloping
  algebras, and invariant theory ({P}aris, 1989)}, volume~92 of {\em Progr.
  Math.}, pages 125--158. Birkh\"auser Boston, Boston, MA, 1990.

\bibitem[Kly83]{Klyachko}
A.~A. Klyachko.
\newblock Models for complex representations of groups {${\rm GL}(n,\,q)$}.
\newblock {\em Mat. Sb. (N.S.)}, 120(162)(3):371--386, 1983.

\bibitem[KP84]{KP}
D.~A. Kazhdan and S.~J. Patterson.
\newblock Metaplectic forms.
\newblock {\em Inst. Hautes \'Etudes Sci. Publ. Math.}, 59:35--142, 1984.

\bibitem[KPW02]{KP3}
D.~Kazhdan, B.~Pioline, and A.~Waldron.
\newblock Minimal representations, spherical vectors and exceptional theta
  series.
\newblock {\em Comm. Math. Phys.}, 226(1):1--40, 2002.

\bibitem[KS90]{KS}
D.~Kazhdan and G.~Savin.
\newblock The smallest representation of simply laced groups.
\newblock In {\em Festschrift in honor of {I}. {I}. {P}iatetski-{S}hapiro on
  the occasion of his sixtieth birthday, {P}art {I} ({R}amat {A}viv, 1989)},
  volume~2 of {\em Israel Math. Conf. Proc.}, pages 209--223. Weizmann,
  Jerusalem, 1990.

\bibitem[LM13]{LM3}
E.~Lapid and Z.~Mao.
\newblock {W}hittaker-{F}ourier coefficients of cusp forms on
  $\widetilde{Sp}_n$: reduction to a local statement, 2013.
\newblock arXiv:1401.0198 [math.NT], {http://arxiv.org/pdf/1401.0198v3.pdf}.

\bibitem[LM14a]{LM6}
E.~Lapid and Z.~Mao.
\newblock Model transition for representations of metaplectic type, 2014.
\newblock to appear in Internat. Math. Res. Notices.

\bibitem[LM14b]{LM4}
E.~Lapid and Z.~Mao.
\newblock On an analogue of the {I}chino--{I}keda conjecture for {W}hittaker
  coefficients on the metaplectic group, 2014.
\newblock arXiv:1404.2905 [math.NT], {http://arxiv.org/pdf/1404.2905v2.pdf}.

\bibitem[LM15]{LM5}
E.~Lapid and Z.~Mao.
\newblock A conjecture on {W}hittaker-–{F}ourier coefficients of cusp forms.
\newblock {\em J. Number Theory}, 146:448--–505, 2015.

\bibitem[LS08]{LokeSavin}
H.~Y. Loke and G.~Savin.
\newblock On minimal representations of {C}hevalley groups of type {$D_n,\
  E_n$} and {$G_2$}.
\newblock {\em Math. Ann.}, 340(1):195--208, 2008.

\bibitem[LS10]{LokeSavin2}
H.~Y. Loke and G.~Savin.
\newblock Modular forms on non-linear double covers of algebraic groups.
\newblock {\em Trans. Amer. Math. Soc.}, 362(9):4901--4920, 2010.

\bibitem[Mat69]{Mats}
H.~Matsumoto.
\newblock Sur les sous-groupes arithm\'etiques des groupes semi-simples
  d\'eploy\'es.
\newblock {\em Ann. Sci. \'Ecole Norm. Sup. (4)}, 2(1):1--62, 1969.

\bibitem[Mat09]{Mt}
I.~Mati{\'c}.
\newblock Levi subgroups of {$p$}-adic {${\rm Spin}(2n+1)$}.
\newblock {\em Math. Commun.}, 14(2):223--233, 2009.

\bibitem[Mez04]{Mezo}
P.~Mezo.
\newblock Metaplectic tensor products for irreducible representations.
\newblock {\em Pacific J. Math.}, 215(1):85--96, 2004.

\bibitem[Mui04]{Mu2}
G.~Mui\'{c}.
\newblock Reducibility of standard representations for classical ${p}$-adic
  groups.
\newblock In {\em Functional analysis VIII (Dubrovnik, 2003)}, volume~47 of
  {\em Various Publ. Ser.}, pages 132--–145. Univ. Aarhus, Aarhus, 2004.

\bibitem[MVW87]{MVW}
C.~M{\oe}glin, M.-F. Vign\'{e}ras, and J.-L. Waldspurger.
\newblock {\em Correspondances de {H}owe sur un corps {$p$}-adique}, volume
  1291 of {\em Lecture Notes in Mathematics}.
\newblock Springer-Verlag, Berlin, 1987.

\bibitem[OS08]{OS}
O.~Offen and E.~Sayag.
\newblock Uniqueness and disjointness of {K}lyachko models.
\newblock {\em J. Funct. Anal.}, 254(11):2846--2865, 2008.

\bibitem[Per81]{P}
P.~Perrin.
\newblock Repr\'{e}sentations de {S}chr\"{o}dinger, indice de {M}aslov et
  groupe metaplectique.
\newblock In {\em Non Commutative Harmonic Analysis and Lie Groups (Proc.
  Marseille-Luminy 1980)}, volume 880 of {\em Lecture Notes in Math.}, pages
  370–--407. Springer-Verlag, 1981.

\bibitem[PPS89]{PPS}
S.~J. Patterson and I.~I. Piatetski-Shapiro.
\newblock The symmetric-square {$L$}-function attached to a cuspidal
  automorphic representation of {${\rm GL}_3$}.
\newblock {\em Math. Ann.}, 283(4):551--572, 1989.

\bibitem[PR12]{PR}
D.~Prasad and D.~Ramakrishnan.
\newblock Self-dual representations of division algebras and {W}eil groups: a
  contrast.
\newblock {\em Amer. J. Math.}, 134(3):749--767, 2012.

\bibitem[Pra93]{Pd}
D.~Prasad.
\newblock Weil representation, {H}owe duality, and the theta correspondence.
\newblock In {\em Theta functions: from the classical to the modern}, volume~1
  of {\em CRM Proc. Lecture Notes}, pages 105--127. Amer. Math. Soc.,
  Providence, RI, 1993.

\bibitem[Rao93]{Rao}
R.~Rao.
\newblock On some explicit formulas in the theory of {W}eil representations.
\newblock {\em Pacific J. Math.}, 157(2):335–--371, 1993.

\bibitem[Ros96]{Roskies}
J.~R. Roskies.
\newblock {\em The minimal representation of {SO(4,3)} over a {P}-adic field}.
\newblock Stanford University, California, 1996.

\bibitem[Sab96]{Sabourin}
H.~Sabourin.
\newblock Une repr\'esentation unipotente associ\'ee \`a l'orbite minimale: le
  cas de {${\rm SO}(4,3)$}.
\newblock {\em J. Funct. Anal.}, 137(2):394--465, 1996.

\bibitem[Sav92]{Savin3}
G.~Savin.
\newblock On the tensor product of theta representations of {${\rm GL}_3$}.
\newblock {\em Pacific J. Math.}, 154(2):369--380, 1992.

\bibitem[Sav93]{Savin2}
G.~Savin.
\newblock An analogue of the {W}eil representation for {$G_2$}.
\newblock {\em J. Reine Angew. Math.}, 434:115--126, 1993.

\bibitem[Sav94]{Savin}
G.~Savin.
\newblock Dual pair {$G_{\mathcal{J}}\times{\rm PGL}_2$}, {$G_{\mathcal{J}}$}
  is the automorphism group of the {J}ordan algebra {${\mathcal{J}}$}.
\newblock {\em Invent. Math.}, 118(1):141--160, 1994.

\bibitem[Sha92]{Sh5}
F.~Shahidi.
\newblock Twisted endoscopy and reducibility of induced representations for
  ${p}$-adic groups.
\newblock {\em Duke Math. J.}, 66(1):1--41, 1992.

\bibitem[Sou05]{Soudry5}
D.~Soudry.
\newblock On {L}anglands functoriality from classical groups to ${GL_{n}}$.
\newblock In J.~Tilouine, H.~Carayol, M.~Harris, and M.-F. Vign\'{e}ras,
  editors, {\em Formes Automorphes (I); Actes du Semestre du CEB, Printemps
  2000}, volume 298 of {\em Ast\'{e}risque}, pages 335--390, 2005.

\bibitem[Sou06]{Soudry4}
D.~Soudry.
\newblock {R}ankin--{S}elberg integrals, the descent method, and {L}anglands
  functoriality.
\newblock In {\em ICM}, pages 1311--1325, Madrid, 2006, 2006. EMS.

\bibitem[Sun97]{Su2}
H.~Sun.
\newblock Examples of metaplectic tensor products, 1997.
\newblock preprint, available at
  {http://www.geocities.ws/mathtester/pubs/tensor.pdf}, last accessed December
  2013.

\bibitem[Tak13]{Tk2}
S.~Takeda.
\newblock Metaplectic tensor products for automorphic representations of
  {GL˜(r)}, 2013.
\newblock arXiv:1303.2785 [math.RT], {http://arxiv.org/abs/1303.2785}.

\bibitem[Tak14]{Tk}
S.~Takeda.
\newblock The twisted symmetric square {L}-function of {GL(r)}.
\newblock {\em Duke Math. J.}, 163(1):175--266, 2014.

\bibitem[Tor97]{Torasso}
P.~Torasso.
\newblock M\'ethode des orbites de {K}irillov-{D}uflo et repr\'esentations
  minimales des groupes simples sur un corps local de caract\'eristique nulle.
\newblock {\em Duke Math. J.}, 90(2):261--377, 1997.

\bibitem[vD78]{Dijk}
G.~van Dijk.
\newblock Smooth and admissible representations of {$p$}-adic unipotent groups.
\newblock {\em Compositio Math.}, 37(1):77--101, 1978.

\bibitem[Vog81]{V}
D.~A. Vogan, Jr.
\newblock Singular unitary representations.
\newblock In {\em Noncommutative harmonic analysis and {L}ie groups
  ({M}arseille, 1980)}, volume 880 of {\em Lecture Notes in Math.}, pages
  506--535. Springer, Berlin, 1981.

\bibitem[Wei64]{We}
A.~Weil.
\newblock Sur certains groupes d'op\'erateurs unitaires.
\newblock {\em Acta Math.}, 111(1):143--211, 1964.

\end{thebibliography}

\end{document}